\documentclass[preprint,12pt]{elsarticle}
\pdfoutput=1  
\usepackage[english]{babel}
\usepackage{amsmath,amssymb,amsfonts,amsthm,bm,mathrsfs}
\usepackage{graphicx,xcolor,colortbl,array,booktabs}
\usepackage{natbib,url,setspace,float,soul}
\usepackage{caption,subcaption}
\usepackage{enumitem}
\usepackage{tikz}
\usetikzlibrary{patterns,arrows.meta}
\usepackage{hyperref}
\usepackage{cleveref}

\setlist[itemize]{topsep=4pt,itemsep=1pt,parsep=0pt,partopsep=0pt}
\AtBeginDocument{%
  \setlength{\abovedisplayskip}{8pt plus 2pt minus 2pt}%
  \setlength{\belowdisplayskip}{8pt plus 2pt minus 2pt}%
  \setlength{\abovedisplayshortskip}{3pt plus 2pt}%
  \setlength{\belowdisplayshortskip}{5pt plus 2pt minus 2pt}%
  \setlength{\jot}{2pt}%
  \raggedbottom
}

\newtheorem{definition}{Definition}
\newtheorem{remark}{Remark}
\newtheorem{theorem}{Theorem}

\newtheorem{lemma}{Lemma}
\newtheorem{proposition}{Proposition}

\newtheorem{corollary}{Corollary}

\newcommand{\RR}{\mathbb{R}}
\newcommand{\Sd}[1][d-1]{\mathbb{S}^{#1}}

\newcommand{\nat}{\mathbb{N}}

\newcommand{\ind}{\boldsymbol{1}}
\newcommand{\pr}{\mathbb{P}}

\newcommand{\E}{\mathbb{E}}

\newcommand{\I}{\mathrm{I}}
\newcommand{\p}{\mathbb{P}}
\newcommand{\eqd}{\stackrel{d}{=}}
\newcommand{\tr}{^\top}
\newcommand{\bb}{{B}}

\journal{Journal of Multivariate Analysis}

\begin{document}

\begin{frontmatter}

\title{Estimating axial symmetry using random projections}

\author[1]{Alejandro Cholaquidis}
\author[1]{Ricardo Fraiman}
\author[3]{Manuel Hernández-Banadik\corref{mycorrespondingauthor}}
\author[4]{Stanislav Nagy}

\address[1]{Facultad de Ciencias, Universidad de la Rep\'ublica}
 \address[3]{Facultad de Ciencias Económicas y Administración, Universidad de la Rep\'ublica}
\address[4]{Faculty of Mathematics and Physics, Charles University}

\cortext[mycorrespondingauthor]{Corresponding author. Email address: \url{m.hernandez.banadik@gmail.com}}

        \begin{abstract}
            We study the directions of axial symmetry of multivariate distributions and relate the measure or cardinality of this set to spherical symmetry. Under Carleman's condition, two independent random projections identify all true axes of symmetry in \(\RR^2\) almost surely. 
            We propose a level-set estimator of the symmetry axes and prove its Hausdorff consistency, in arbitrary dimension, for the zero set of the population criterion based on finitely many projections. Together with planar identification, this gives a consistent estimation of the true axes of symmetry in dimension two.
        \end{abstract}

\begin{keyword}
axial symmetry \sep multivariate symmetry \sep random projections \sep Cram\'er-Wold theorem \sep symmetry estimation

\MSC[2020] 62H12 \sep 62G05
\end{keyword}

\end{frontmatter}


     \section{Introduction}

        Multivariate symmetry is defined by invariance under a group of transformations; see \cite{serfling2006multivariate} for an overview. In medical imaging, departures from axial symmetry can help identify changes in an organ's shape. Examples include the analysis of brain magnetic resonance images \cite{martos2018discrimination} and lung asymmetry in chest radiographs \cite{hogeweg2017fast}. CAD4TB version 3, a software package for computer-aided tuberculosis detection, includes an estimated symmetry axis in its analysis \citep{murphy2020computer}.

        Symmetry also restricts the class of objects to be estimated and can simplify set and density estimation. Axial and related symmetries arise in directional statistics; see \cite{pewsey2021recent} and the references therein.

We first study how the measure or cardinality of the set of axial symmetry directions implies spherical symmetry of a distribution or its marginals. We then seek a Cram\'er--Wold device for identifying symmetry axes from one-dimensional projections, by comparing the projected laws of a random vector and its transformation.

The extension of the Cram\'er--Wold theorem \citep{Cramer_Wold1936} by \citet{Cuesta_etal2007} implies, under a moment condition, that two distinct distributions have identical projections only on a null set of directions. For a fixed transformation, this suffices for projection-based tests of symmetry or invariance; see \cite{fraiman2024application, fraiman2017some}. Estimating an unknown axis requires a statement that holds simultaneously for every candidate axis. The exceptional null set may depend on the axis, so the fixed-transformation result alone does not settle this question.

We resolve this problem in dimension two: under Carleman's condition, two independent uniform projection directions identify all symmetry axes almost surely. Thus, the recovered set does not depend on the sampled directions. We conjecture that \(d\) independent uniform projections suffice in \(\RR^d\); the case \(d>2\) remains open.

We use this identification result to construct an estimator of the symmetry directions. To our knowledge, this is the first method proposed for this task. In arbitrary dimension, the level-set estimator converges in Hausdorff distance to the zero set of the population criterion for the chosen projections. Planar identification then gives consistency for the true axes when $d=2$. The same conclusion for $d>2$ would follow from our conjecture.

Section~\ref{sec: setup} introduces axial symmetry and the projection criterion. Section~\ref{sec: theory} relates symmetry directions to spherical symmetry (Section~\ref{sec: carleman}), proves planar identification (Theorem~\ref{theorem: dimension 2}), and gives a counterexample without moment assumptions (Section~\ref{sec: counterexample}). Section~\ref{sec: estimator} establishes level-set consistency (Theorem~\ref{teo:consistency}) and its planar consequence (Corollary~\ref{cor:consistency-d2}), and discusses testing and an unknown center. Sections~\ref{sec:simulation} and~\ref{sec:aplicacion} present simulations and an application to chest radiographs.

	\section{Setup and motivation}\label{sec: setup}

	Let $X$ be a random vector in $\RR^d$, $d\geq2$, with the Euclidean inner product. Write $\Sd=\{u\in\RR^d:\|u\|=1\}$ and let $\nu$ be the uniform distribution on $\Sd$. We assume that $\E(X)$ exists and is finite, except in Section~\ref{sec: counterexample}.

\begin{definition}\label{def:symmetry}
A random vector $X\in\RR^d$ is axially symmetric with respect to $u\in\Sd$ if $X-\E(X)\eqd R_u(X-\E(X))$, where $R_u=2uu\tr-\I_d$. Here $\eqd$ denotes equality in distribution, $u\tr$ denotes the transpose of $u$, and $\I_d$ is the $d\times d$ identity matrix.
\end{definition}

The orthogonal map $R_u$ fixes the line $\operatorname{span}(u) = \left\{\lambda u \colon \lambda \in \RR \right\}$ and reverses orthogonal vectors $u^\perp$: it is reflects $u^\perp$ across $\operatorname{span}(u)$ when $d=2$ and rotates it by $\pi$ about $\operatorname{span}(u)$ when $d=3$. For example, if $u=(1,0,\ldots,0)\tr$, then
\begin{equation*}
R_u=\begin{pmatrix}1&0_{d-1}\tr\\0_{d-1}&-\I_{d-1}\end{pmatrix},
\end{equation*}
where $0_{d-1}$ is the zero vector in $\RR^{d-1}$.

Axial symmetry of $X$ about a specified center $c\in\RR^d$ means $X-c\eqd R_u(X-c)$. The choice $c=\E(X)$ gives Definition~\ref{def:symmetry}; the formulation with a specified center will be needed in Section~\ref{sec: counterexample}, where the mean does not exist.

We assume $\E(X)=0_d$ to simplify notation.\footnote{In Sections~\ref{sec: setup} and~\ref{sec: theory}, the center is known, so this is a normalization. Estimation with an unknown center will be addressed in Section~\ref{sec: estimator}, Remark~\ref{rem:center}.} The set of symmetry directions is
\begin{equation}\label{eq: U}
\mathcal U=\{u\in\Sd:X\eqd R_uX\}.
\end{equation}

For a real random variable $Y$, let $F_Y$ denote its distribution function. For $h\in\Sd$, define
\[
g_h(u)=\sup_{t\in\RR}\left|F_{\langle X,h\rangle}(t)-F_{\langle R_uX,h\rangle}(t)\right|
\]
and $\mathcal U^h=\{u\in\Sd:h\tr X\eqd h\tr R_uX\}$ $=\{u\in\Sd:g_h(u)=0\}$. For $\mathfrak H=\{h_1,\ldots,h_k\}\subset\Sd$, $k\geq1$, put
\begin{align}
g_{\mathfrak H}(u) :=\frac1k\sum_{j=1}^k g_{h_j}(u), \quad \mbox{and} \quad \mathcal U^{\mathfrak H} :=\{u\in\Sd:g_{\mathfrak H}(u)=0\}. \label{eq: gH} 
\end{align}
Since equality in distribution implies equality of every projected law, $\mathcal U\subseteq\mathcal U^{\mathfrak H}=\bigcap_{j=1}^k\mathcal U^{h_j}$.

Our estimator uses $\mathcal U^{\mathfrak H}$ to recover $\mathcal U$ under a moment condition, through a Cram\'er--Wold device \cite{Cramer_Wold1936,Cuesta_etal2007,fraiman2024application}. The inclusion $\mathcal U\subseteq\mathcal U^{\mathfrak H}$ can be strict: $g_h(h)=0$ for every $h$, and central symmetry (that is, $X \eqd -X$) gives $g_h(v)=0$ whenever $h\perp v$, even if $v$ is not a symmetry axis of $X$. Our key Theorem~\ref{theorem: dimension 2} below will show that, in dimension two and under Carleman's condition, two independent directions $\mathfrak H=\{h_1, h_2\}$ drawn from $\nu$ give $\mathcal U^{\mathfrak H}=\mathcal U$ almost surely.

	\section{Theoretical results}  \label{sec: theory}

\subsection{On symmetry under Carleman's condition}\label{sec: carleman}
    Our arguments require one-dimensional projections to be determined by their moments. Carleman's condition is a sufficient condition that we impose on $\|X\|$ to ensure this for every $h\tr X$; any assumption ensuring moment determinacy of all projections could replace it. Finiteness of all moments alone is insufficient for moment determinacy, as lognormal laws show. The bound $|h\tr X|\leq\|X\|$ gives $(\E|h\tr X|^{2k})^{1/(2k)}\leq(\E\|X\|^{2k})^{1/(2k)}$, so each projection's even-moment Carleman series dominates that of $\|X\|$. The divergence of the latter series therefore implies moment determinacy \citep{Carleman1922}. Combined with the Cram\'er--Wold theorem \citep{Cramer_Wold1936}, this allows us to deduce symmetry from equality of projected laws.

\begin{definition}
Let $X \in \RR^d$ be a random vector with moments of all orders
$m_k \;=\; \mathbb{E}\big[ \left\Vert X \right\Vert^k \big],$ $k \in \mathbb{N}.$
We say that $X$ satisfies the Carleman condition if $\sum_{k=1}^{\infty} m_{k}^{-\frac{1}{k}} \;=\; \infty.$
\end{definition}

By Lyapunov's inequality \citep{Billingsley1995ProbabilityMeasure}, $a_k:=m_k^{-1/k}$ is non-increasing. Hence $\sum_{j=1}^{2N+1}a_j\leq a_1+2\sum_{k=1}^Na_{2k}$ and $\sum_ka_{2k}\leq\sum_ka_k$, so the Carleman condition is equivalent to the even-moment form
$\sum_{k=1}^{\infty}(\E\|X\|^{2k})^{-1/(2k)}=\infty$
used below.

A random vector $X\in\RR^d$ is spherically symmetric if $AX\eqd X$ for every orthogonal matrix $A$; its distribution is then called spherical.

\begin{proposition} \label{proposition: spherical symmetry}
Let $X\in \RR^2$ be a random vector that fulfills the Carleman condition.
Assume there exists a real random variable $T$ such that for infinitely many unit directions
$h\in\mathbb{S}^1$ one has
$h^{\top}X \stackrel{d}{=} T$.
Then $X$ is spherically symmetric.
In particular, $h^{\top}X\stackrel{d}{=}T$ for every $h\in\mathbb{S}^1$, and $T$ is symmetric about~$0$.
 \end{proposition}

\begin{proof} For each $k\in\mathbb{N}$ define $P_k(h)\;=\;\mathbb{E}\!\big[(h^{\top}X)^{k}\big]$, $\|h\|=1.$
Clearly $P_k$ is a homogeneous polynomial of degree $k$ in the coordinates of $h$; when restricted to
$\mathbb{S}^1$ it is a trigonometric polynomial (hence real-analytic in the angle).
By hypothesis, there is an infinite set $H\subseteq\mathbb{S}^1$ such that
$P_k(h)=b_k:=\mathbb{E}[T^{k}]$ for every $h\in H$.
By compactness, $H$ has an accumulation point. The identity theorem for real-analytic functions \cite[Corollary 1.2.7]{KrantzParks2002} applied to $\theta\mapsto P_k((\cos\theta,\sin\theta)^\top)-b_k$ therefore gives $P_k(h)=b_k$ for every $h\in\mathbb S^1$.
Thus every projection has the same moments as $T$. For every unit vector $h$,
\[
m_{2k}(h) := \mathbb{E}\!\left[\,|h^\top X|^{2k}\,\right]
\;\le\; \mathbb{E}\!\left[\|X\|^{2k}\right] =:  M_{2k}
\quad\Longrightarrow\quad
m_{2k}(h)^{-\frac{1}{2k}} \ge M_{2k}^{-\frac{1}{2k}}.
\]
Each projection therefore satisfies Carleman's condition and is moment-determinate, giving $h^\top X\eqd T$ for every $h\in\mathbb S^1$.

Now let $Q$ be a $2 \times 2$ orthogonal matrix.  For every $h\in\mathbb{S}^1$, $h^{\top}(QX)\;=\;(Q^{\top}h)^{\top}X \stackrel{d}{=} T.$
The Cram\'er--Wold theorem gives $QX\eqd X$, proving spherical symmetry. In particular, $-X\eqd X$, so $T\eqd-T$.
\end{proof}

Without Carleman's condition, Proposition~\ref{proposition: spherical symmetry} can fail \cite[Example~24]{Ranosova2023}.

\begin{corollary}
Let $P$ be a distribution on $\RR^2$ satisfying Carleman's condition. If $\mathcal U$ in~\eqref{eq: U} is infinite, then $P$ is spherical.
\end{corollary}

\begin{proof}
Fix $h\in\mathbb S^1$ and set $T=h\tr X$. For $u\in\mathcal U$,
$h\tr X\eqd(R_uh)\tr X$, so $R_uh\in H(h):=\{g\in\mathbb S^1:g\tr X\eqd T\}$.
If $R_{u_1}h=R_{u_2}h$, the planar rotation $R_{u_2}R_{u_1}$ fixes $h$. This happens only if $R_{u_2}R_{u_1} = \I_d$, which implies $u_2=\pm u_1$. Thus $u\mapsto R_uh$ is at most two-to-one. Since $\mathcal U$ is infinite, so is $H(h)$, and Proposition~\ref{proposition: spherical symmetry} proves the claim.
\end{proof}

    \begin{remark}\label{rem:gauss3d}
	The preceding conclusions valid in the plane fail in $\RR^3$. Indeed, if $X$ is centered Gaussian with covariance $\operatorname{diag}(1+\epsilon,1,1)$, $\epsilon>0$, every unit direction in the plane $x_1=0$ is an axis of symmetry and has projected law $N(0,1)$, although $X$ is not spherical.
    \end{remark}

Finite symmetry groups need not imply spherical symmetry. The next proposition shows that a set of symmetry directions of positive surface measure does, in any dimension $d\geq2$ and without a Carleman assumption.

	\begin{proposition}\label{prop:positive_measure_spherical}
	Let $X \in \RR^d$, $d\geq2$, be a random vector. If the set of directions of symmetry $\mathcal{U}$ of $X$ has positive $\nu$-measure, then the distribution of $X$ is spherically symmetric.
	\end{proposition}

We first establish two properties of $\mathcal U$: stability under conjugation and closedness.

\begin{lemma}\label{lem:composition_symmetries}
	Let \(u_1, u_2 \in \Sd\), $d\geq 2$. If \(X \in \RR^d\) { is a random vector} symmetric about the directions generated by \(u_1\) and \(u_2\), then it is also symmetric about the direction \(v = R_{u_1} u_2\).
\end{lemma}

\begin{proof}
Since $v=R_{u_1}u_2$ and $R_{u_1}$ is orthogonal and symmetric,
	\[
	R_v =  2 (R_{u_1} u_2)(R_{u_1} u_2)^\top - \I_d = R_{u_1} (2 u_2 u_2^\top - \I_d) R_{u_1} = R_{u_1} R_{u_2} R_{u_1}.
	\]
	The two assumed symmetries give $R_vX=R_{u_1}R_{u_2}R_{u_1}X\eqd X$.
\end{proof}

   \begin{lemma}   \label{lemma: closed}
        The set of directions of symmetry $\mathcal{U}$ of a random vector $X \in \RR^d$ is closed.
    \end{lemma}

\begin{proof}
    Let \(\{u_\ell\}_{\ell \in \mathbb{N}}\) be a sequence in \(\Sd\) converging to \(u_0\) such that \(X \eqd R_{u_\ell}X\) for all \(\ell\). The result follows from the convergence \(R_{u_\ell} \to R_{u_0}\) as $\ell \to \infty$, and the continuous mapping theorem \citep[Theorem~5.5]{Billingsley1968}.
\end{proof}
\begin{proof}[Proof of Proposition~\ref{prop:positive_measure_spherical}]
We first show that $\mathcal U=\Sd$. By
Lemma~\ref{lemma: closed}, $\mathcal U$ is closed. Since
$\nu(\mathcal U)>0$, the Lebesgue density theorem on the sphere
provides a density point $x_0\in\mathcal U$
\citep[see][Corollary~261D, applied in local coordinates]{Fremlin2003}.
Suppose, for contradiction, that $\mathcal U\neq\Sd$. Then there
exist $y\in\Sd$ and $0<\epsilon<\pi$ such that
\(\bb(y,\epsilon)\cap\mathcal U=\emptyset,\)
where $\bb(y, \epsilon)$ is the open ball centered at $y$ with radius $\epsilon$ for the geodesic distance $d_g$.

Let \(T=\{v\in\Sd: \langle v,x_0\rangle=0\}\)
be the unit sphere in the hyperplane orthogonal to $x_0$, and let $\sigma$ be its
uniform probability measure. When $d=2$, $T$ consists of two
points, each with mass $1/2$. For $v\in T$, write
\[
\gamma_v=\{\cos(t)x_0+\sin(t)v:t\in\RR\},
\qquad
\Gamma=\{v\in T:\gamma_v\cap\bb(y,\epsilon/2)\neq\emptyset\}.
\]
The set $\Gamma$ is nonempty, because some great circle passes
through $x_0$ and $y$, and it is open by continuity.
Consequently, $c:=\sigma(\Gamma)>0$.

For $0<r<\pi$, consider the geodesic cone
\[
C_r=\{\cos(t)x_0+\sin(t)v:v\in\Gamma,\ 0<t<r\}.
\]
Geodesic polar coordinates give $\nu(C_r)=c\,\nu(\bb(x_0,r))$.
Since $x_0$ is a density point of $\mathcal U$, we can choose
$0<r<\epsilon/4$ such that
\(\nu\bigl(\bb(x_0,r)\setminus\mathcal U\bigr)
<c\,\nu(\bb(x_0,r))=\nu(C_r).\)
Thus $\mathcal U\cap C_r\neq\emptyset$. Choose
$x_1\in\mathcal U\cap C_r$ and set
\(\alpha=d_g(x_0,x_1),\) \(0<\alpha<r<\epsilon/4.\)
The plane $L=\operatorname{span}\{x_0,x_1\}$ is two-dimensional,
and the great circle $\gamma=\Sd\cap L$ meets
$\bb(y,\epsilon/2)$ by construction.

For every $x\in\gamma$, the restriction of $R_x$ to $L$ is the
planar reflection across the line $\operatorname{span}(x)$. Parametrize $\gamma$
by angle so that $x_0$ and $x_1$ have angles $0$ and $\alpha$,
respectively, and define
\(x_{\ell+1}=R_{x_\ell}x_{\ell-1},\)  \(\ell\geq1.\)
Lemma~\ref{lem:composition_symmetries} and induction show that
$x_\ell\in\mathcal U$ for every $\ell\geq0$. Moreover, reflection
across the line of angle $\ell\alpha$ sends angle
$(\ell-1)\alpha$ to $(\ell+1)\alpha$, so $x_\ell$ has angle
$\ell\alpha$ modulo $2\pi$.

Choose $z\in\gamma\cap\bb(y,\epsilon/2)$, and let
$\theta\in[0,2\pi)$ be its angle. Taking
$\ell=\lfloor\theta/\alpha\rfloor$ gives
\(0\leq\theta-\ell\alpha<\alpha,\)
and hence $d_g(z,x_\ell)<\alpha$. Therefore,
\[
d_g(y,x_\ell)
\leq d_g(y,z)+d_g(z,x_\ell)
<\epsilon/2+\alpha<\epsilon,
\]
contradicting
$\bb(y,\epsilon)\cap\mathcal U=\emptyset$.
It follows that $\mathcal U=\Sd$.

Finally, let $\psi_X$ be the characteristic function of $X$.
Since every $u\in\Sd$ belongs to $\mathcal U$,
\(\psi_X(R_ut)=\psi_X(t),\) \(t\in\RR^d.\)
For any $s,t\in\RR^d$ with $\|s\|=\|t\|$, we can choose $u$
such that $R_ut=s$: if $s\neq-t$, take
$u=(s+t)/\|s+t\|$; if $s=-t\neq0$, take any unit vector
$u\perp t$, which exists because $d\geq2$.
The case $s=t=0$ is immediate.
Thus $\psi_X(s)=\psi_X(t)$ whenever $\|s\|=\|t\|$, so
$\psi_X$ is radial and $X$ is spherically symmetric \cite[Theorem~2.1]{Fang_etal1990}.
\end{proof}
The next theorem identifies all planar symmetry axes from two random projections. Its proof combines the finiteness consequence of Proposition~\ref{proposition: spherical symmetry} with the Cram\'er--Wold theorem of \citet{Cuesta_etal2007} and Fubini's theorem.

\begin{theorem} \label{theorem: dimension 2}
Let $X\in\RR^2$ satisfy Carleman's condition and suppose that $X$ is not spherically symmetric. If $h_1,h_2$ are independent directions with distribution $\nu$, then, almost surely, the solutions $u\in\Sd[1]$ of $h_j^\top X\eqd h_j^\top R_uX$, $j=1,2$, form exactly the (possibly empty) set $\mathcal U$ of symmetry axes.
\end{theorem}
\begin{proof}
For $u\in\Sd[1]$, let
$M_u=\{h\in\Sd[1]:h^\top R_uX\eqd h^\top X\}$.
Fix $h\in\Sd[1]$. Since $X$ is not spherically symmetric,
Proposition~\ref{proposition: spherical symmetry} implies that
$H(h)=\{g\in\Sd[1]:g^\top X\eqd h^\top X\}$ is finite.
Moreover, $h\in M_u$ if and only if $R_uh\in H(h)$.
For each $g\in H(h)$, the equation $R_uh=g$, equivalently
$2(u^\top h)u=h+g$, has exactly two solutions:
$u=\pm(h+g)/\|h+g\|$ if $g\neq-h$, and the two unit vectors
orthogonal to $h$ if $g=-h$.
Thus each fixed $h$ belongs to at most $2|H(h)|$ (here, $|H(h)|$ is the number of elements of $H(h)$) of the sets $M_u$.

Let $\psi_X$ denote the characteristic function of $X$.
The condition $h\in M_u$ is equivalent to
$\psi_X(tR_uh)=\psi_X(th)$ for every $t\in\RR$.
For each fixed $t$, both sides are continuous in $(u,h)$,
so the set $\{(u,h)\in\Sd[1]\times\Sd[1]:h\in M_u\}$ is closed; in particular, every $M_u$ is closed. Furthermore, if $u\notin\mathcal U$, then
$\nu(M_u)=0$ by \cite[Corollary~3.2]{Cuesta_etal2007}.
Let $M=\bigcup_{u\notin\mathcal U}(M_u\times M_u)$ be the set of projection pairs that admit a false axis.

We verify that $M$ is Borel before applying Fubini's theorem.
Since $\mathcal U$ is closed by Lemma~\ref{lemma: closed},
write $\Sd[1]\setminus\mathcal U=\bigcup_{m\geq1}K_m$ with
each $K_m$ compact. For every $m$, the set
$B_m=\{(u,h_1,h_2)\in K_m\times(\Sd[1])^2:
h_1,h_2\in M_u\}$ is closed by the characteristic-function
argument above, and hence compact.
Its image under the continuous projection
$\Pi(u,h_1,h_2)=(h_1,h_2)$ is therefore compact.
Since $M=\bigcup_{m\geq1}\Pi(B_m)$, it follows that $M$
is a countable union of compact sets, hence Borel.

For every $h\in\Sd[1]$, the section
\begin{equation}\label{eq: Mh}
M_h=\{h_2\in\Sd[1]:(h,h_2)\in M\}
   =\bigcup_{\substack{u\notin\mathcal U\\h\in M_u}}M_u
\end{equation}
is a finite union of $\nu$-null Borel sets, by the finiteness
established above. Consequently, Fubini's theorem
\citep[Theorem~4.4.5]{Dudley2002} gives
\[
\nu^2(M)=\int_{\Sd[1]}\nu(M_h)\,\mathrm{d}\nu(h)=0.
\]
Thus, for almost every pair $(h_1,h_2)$ drawn independently
from $\nu$, no false axis $u$ satisfies both equalities $h_j^\top X\eqd h_j^\top R_uX$, $j=1,2$.
Every true axis satisfies these equalities for all projection
directions, so the solution set coincides with $\mathcal U$
almost surely.
\end{proof}
The proof does not extend directly to $d>2$: the Gaussian example in \Cref{rem:gauss3d} shows that $H(h)$ can be infinite even for a non-spherical law satisfying Carleman's condition. The sections $M_h$ in~\eqref{eq: Mh} may then involve uncountable unions of null sets, which need not be null. We conjecture that $d$ independent random projections suffice for identification in $\RR^d$.

	\subsection{A counterexample without moment assumptions}
\label{sec: counterexample}

We construct a distribution for which two random projections
accept a false symmetry axis with positive probability.
Throughout this subsection, symmetry is understood with
respect to the specified center $0_2$.

Let $C_1,C_2$ be independent standard Cauchy random variables
and set $X=\frac13(C_1+2C_2,C_1-C_2)^\top$.
Consider the centrally symmetric parallelogram
\begin{equation*}
K=\operatorname{conv}\{\pm p,\pm q\},
\qquad p=(1,-1)^\top,\quad q=(1,2)^\top.
\end{equation*}
Its associated norm, defined by
$\|z\|_K=\inf\{\lambda>0:z\in\lambda K\}$, satisfies
$\|(s,t)^\top\|_K=\max\{|s|,|2t-s|/3\}
=(|s+t|+|2s-t|)/3$.
Independence and the characteristic function of the standard
Cauchy law give
\begin{equation*}
\psi_X(s,t)
=\exp\left\{-\frac{|s+t|+|2s-t|}{3}\right\}
=e^{-\|(s,t)^\top\|_K}.
\end{equation*}
Consequently, for any $h\in\Sd[1]$, the projection $h^\top X$
is Cauchy with scale $\|h\|_K$. Since $R_u^\top=R_u$,
the projection $h^\top R_uX$ is Cauchy with scale
$\|R_uh\|_K$, so these two projected laws coincide when $\|h\|_K=\|R_uh\|_K$.

Take $u=(0,1)^\top$, so that $R_u(s,t)^\top=(-s,t)^\top$.
We have $\|q\|_K=1$ and $\|R_uq\|_K=5/3$, hence
$\psi_X(q)\neq\psi_X(R_uq)$ and $u\notin\mathcal U$.
Geometrically, $K$ and $R_uK$ are distinct, as shown in
Figure~\ref{fig:K}.

Nevertheless, for every
$h=(a,b)^\top\in H_0:=\{(a,b)^\top\in\Sd[1]:|b|<|a|\}$,
the norm formula gives
$\|h\|_K=\|R_uh\|_K=|a|$.
Thus $g_h(u)=0$ throughout $H_0$, which has uniform measure
$\nu(H_0)=1/2$. For two independent directions
$h_1,h_2$ drawn from $\nu$, with $\mathfrak H=\{h_1,h_2\}$,
it follows that
\[
\p\bigl(\mathcal U^{\mathfrak H}\neq\mathcal U\bigr)
\geq \p(h_1,h_2\in H_0)
=\nu(H_0)^2=\frac14.
\]
The same argument gives a lower bound $2^{-k}$ for any fixed
number $k$ of independent projection directions.

Every nonzero projection is a nondegenerate Cauchy variable, so $X$ has no finite mean and fails Carleman's condition. This example establishes failure of identification about a specified center, but does not isolate the effect of removing Carleman's condition while retaining a finite mean.

\begin{figure}[!htbp]
\centering
\begin{tikzpicture}[scale=.95,>=Latex,font=\small]
  \fill[orange!20]
    (1,-1)--(1,2)--(-1,1)--(-1,-2)--cycle;
  \path[pattern=north east lines,pattern color=red!65]
    (-1,-1)--(-1,2)--(1,1)--(1,-2)--cycle;

  \draw[->,thin,gray]
    (-1.7,0)--(1.8,0) node[right] {$x_1$};
  \draw[->,thin,gray]
    (0,-2.3)--(0,2.4) node[above] {$x_2$};

  \draw[red!70!black,thick]
    (-1,-1)--(-1,2)--(1,1)--(1,-2)--cycle;
  \draw[thick]
    (1,-1)--(1,2)--(-1,1)--(-1,-2)--cycle;

  \draw[very thick,blue] (1,-1)--(1,1);
  \draw[very thick,blue] (-1,-1)--(-1,1);
  \draw[->,thick] (0,0)--(0,1) node[right] {$u$};

  \fill (1,-1) circle (1.2pt) node[right] {$p$};
  \fill (1,2) circle (1.2pt) node[right] {$q$};
  \fill (-1,1) circle (1.2pt) node[left] {$-p$};
  \fill (-1,-2) circle (1.2pt) node[left] {$-q$};

  \node at (.6,1.5) {$K$};
  \node[red!70!black] at (-.6,1.5) {$R_uK$};
\end{tikzpicture}
\caption{The parallelogram $K$ (orange) and its reflection
$R_uK$ (red hatching) across $\operatorname{span}(u)$, with $u=(0,1)^\top$.
Although the bodies $K$ and $R_uK$ differ, their boundaries share the blue
segments, giving equal projected laws along the directions
in $H_0$.}
\label{fig:K}
\end{figure}
		\section{A plug-in based estimator}\label{sec: estimator}

    Let $\aleph_n=\{X_1,\dots,X_n\}$ be an i.i.d. sample from $P$ in $\RR^d$, with known center $\E(X)=0_d$; sample-mean centering will be treated in Remark~\ref{rem:center}. Split the sample into the balanced groups $\aleph_1=\{X_1,\dots,X_{\lfloor n/2\rfloor}\}$ and $\aleph_2=\{X_{\lfloor n/2\rfloor+1},\dots,X_n\}$. A random split independent of the data has the same distribution. Draw $\mathfrak H=\{h_1,\dots,h_k\}$ from $\nu^k$, independently of the sample, once and keep it fixed as $n$ grows. Randomizing the directions provides the identification guarantee in Theorem~\ref{theorem: dimension 2}; we do not claim that one deterministic finite collection works for every distribution. Recording the sampled directions makes the procedure reproducible. We estimate the symmetry axes using $g_{\mathfrak H}$ in~\eqref{eq: gH} and its empirical counterpart:

	\begin{align*}
		F^n_{\langle X, h \rangle}(t) &= \frac{1}{|\aleph_1|} \sum_{X_i \in \aleph_1} \ind\{\langle X_i, h \rangle \leq t\}, \\
		F^n_{\langle R_u X, h \rangle}(t) &= \frac{1}{|\aleph_2|} \sum_{X_i \in \aleph_2} \ind\{\langle R_u X_i, h \rangle \leq t\}, \\
		\widehat g_{n,h}(u) &= \sup_{t \in \RR} \left| F^n_{\langle X, h \rangle}(t) - F^n_{\langle R_u X, h \rangle}(t) \right|, \\
		\widehat g_{n,\mathfrak{H}}(u) &= \frac1k \sum_{j=1}^k \widehat g_{n, h_j}(u).
	\end{align*}
Here $|A|$ denotes cardinality and $\ind\{A\}$ the indicator of $A$. We henceforth omit $\mathfrak H$ from $\widehat g_{n,\mathfrak H}$ and $g_{\mathfrak H}$.

    Minimizing $\widehat g_n$ need not recover every symmetry axis. Uniform convergence gives $\limsup_{n\to\infty}\arg\min\widehat g_n\subseteq\arg\min g$ almost surely \citep[Theorem~7.33]{Rockafellar_Wets1998}, so accumulation points of empirical minimizers minimize $g$. If $g$ has a unique minimizer, every sequence of empirical minimizers converges to it. When $g$ has several minimizers, convergence need not recover all of them. We therefore use a level-set estimator.

    Define
    \begin{align}
\mathcal U_n(\epsilon_n)&:=\{u\in\Sd:\widehat g_n(u)<\epsilon_n\},\label{eq:estimator}\\
\beta_n&:=\frac{\sqrt n}{\log n},\qquad \epsilon_n=1/\beta_n\to0.\label{eq:betan}
\end{align}

    Theorem~\ref{teo:consistency} proves Hausdorff consistency for $\mathcal U^{\mathfrak H}$ in any dimension. In the plane, identification then yields consistency for $\mathcal U$ (Corollary~\ref{cor:consistency-d2}). The proof uses uniform convergence, a standard approach to level-set estimation \citep[Theorem~2.1]{molchanov1998limit}; see also \cite{Cholaquidis:2022:LevelSetDensity}. The rate in~\eqref{eq:betan} does not determine a constant; a choice of a constant appropriate for finite samples is justified in Remark~\ref{rem:calibration}.

	Set $\|f\|_\infty=\sup_x|f(x)|$, with the supremum taken over the domain of $f$.
\begin{proposition}\label{prop:uniform_convergence}
		Given \(\widehat g_n\) and \(g\) as introduced above,
		\[
		\frac{\sqrt n}{a_n}\|\widehat g_n-g\|_\infty \to 0 \text{
        as $n \to \infty$, almost surely,}
		\]
		where \(a_n = \sqrt{\delta_n\log n}\), with \(\delta_n \to \infty\). Convergence is almost sure jointly in the sample and the independently drawn, fixed directions $\mathfrak H$.
	\end{proposition}

	\begin{proof}

		For \(v \in \Sd\) and fixed \(h \in \Sd\), define \(\mathfrak{B}_{v,t,h} = \{ x \in \RR^d : \langle R_v x, h \rangle \leq t \}\), and also $\mathfrak B^0_{t,h} = \{x \in \RR^d : \langle x, h\rangle \leq t\}$. Then, for a deterministic direction $h_0 \in \Sd$,
		\begin{equation*}
             \p\bigl( \langle R_v X, h_0 \rangle \leq t \bigr) = P(\mathfrak{B}_{v,t,h_0}).
		\end{equation*}
		Let $m_1 = \lfloor n/2 \rfloor = |\aleph_1|$ and $m_2 = n - m_1 = |\aleph_2|$, and define the two subsample empirical measures
		\[
        P^{(j)}_{m_j}(A) = \frac{1}{m_j} \sum_{X_i \in \aleph_j} \ind\{X_i \in A\}, \quad A \subseteq \RR^d \ \text{Borel and $j = 1,2$}.
		\]
		Since the split is deterministic and the observations are i.i.d., $\aleph_1$ and $\aleph_2$ are two independent i.i.d.\ samples from $P$, of sizes $m_1$ and $m_2$, respectively. With this notation, the empirical distribution functions defining $\widehat g_{n,h}$ can be written as
		\[
		F^n_{\langle X, h\rangle}(t) = P^{(1)}_{m_1}\bigl(\mathfrak B^0_{t,h}\bigr), \qquad F^n_{\langle R_v X, h\rangle}(t) = P^{(2)}_{m_2}\bigl(\mathfrak B_{v,t,h}\bigr),
		\]
		each subsample entering exactly one of the two terms. We first prove that there exists $0 < C \leq 4 e^8$ such that, for all $n \geq 4(d+2)$,
		\begin{equation}\label{boundconvunif}
			\p \left\{ \|\widehat g_n-g\|_\infty > \epsilon \right\} \leq 2 C ( n^{2d+2} + 1 )\, e^{- (n-1) \epsilon^2 / 4 }.
		\end{equation}

		Conditionally on $\mathfrak H$, for all \(h \in \Sd\) and \(v \in \Sd\), we have
        \begingroup\small
\begin{equation} \label{eq: cont1}
		\begin{aligned}
			\left| \widehat{g}_{n,h}(v) - g_h(v) \right| &= \left| \sup_{t \in \RR} \left| F^n_{\langle X, h \rangle}(t) - F^n_{\langle R_v X, h \rangle}(t) \right| - \sup_{t \in \RR} \left| F_{\langle X, h \rangle}(t) - F_{\langle R_v X, h \rangle}(t) \right| \right| \\
			&\leq \sup_{t \in \RR} \left| \left( F^n_{\langle X, h \rangle}(t) - F_{\langle X, h \rangle}(t) \right) - \left( F^n_{\langle R_v X, h \rangle}(t) - F_{\langle R_v X, h \rangle}(t) \right) \right| \\
			&\leq \sup_{t \in \RR} \left| F^n_{\langle X, h \rangle}(t) - F_{\langle X, h \rangle}(t) \right| + \sup_{t \in \RR} \left| F^n_{\langle R_v X, h \rangle}(t) - F_{\langle R_v X, h \rangle}(t) \right| \\
			&= \sup_{t \in \RR} \left| P^{(1)}_{m_1}( \mathfrak{B}^0_{t,h} ) - P( \mathfrak{B}^0_{t,h} ) \right| + \sup_{t \in \RR} \left| P^{(2)}_{m_2}( \mathfrak{B}_{v,t,h} ) - P( \mathfrak{B}_{v,t,h} ) \right|.
		\end{aligned}
        \end{equation}
\endgroup
    Let $\mathfrak{H} = \{h_1, \dots, h_k\}\subset \Sd$ be fixed, then
	\begin{equation} \label{eq: cont2}
        \begin{aligned}
			\sup_{v\in \Sd}| \widehat{g}_n(v) - g(v)| & = \sup_{v\in \Sd} \ \ \bigg| \frac 1 k \sum_{j=1}^k \bigl(\widehat{g}_{n,h_j}(v) - g_{h_j}(v)\bigr) \bigg| \\
			& \leq \sup_{v\in \Sd} \ \  \frac 1 k \sum_{j=1}^k \ \  |  \widehat{g}_{n,h_j}(v) - g_{h_j}(v) | \\
			& \leq \sup_{v,h \in  \Sd} \ \ |  \widehat{g}_{n,h}(v) - g_{h}(v) |.
		\end{aligned}
        \end{equation}
		Therefore,
        \begin{align*}
			\p \left\{ \|\widehat g_n-g\|_\infty > \epsilon \right\} & = \E_{\mathfrak{H}} \Bigg[\p \left( \|\widehat g_n-g\|_\infty > \epsilon | \mathfrak H \right)
            \Bigg],
		\end{align*}
        and
		\begingroup\small
\begin{align*}
			\p \left\{ \|\widehat g_n-g\|_\infty > \epsilon  \bigg | \mathfrak H \right\} &\leq \p \left\{ \sup_{h \in \Sd, t \in \RR} \left| P^{(1)}_{m_1}( \mathfrak{B}^0_{t,h} ) - P( \mathfrak{B}^0_{t,h} ) \right| > \frac{\epsilon}{2} \right\} \\
			& \hspace*{-1.25em} + \p \left\{ \sup_{v, h \in \Sd, t \in \RR} \left| P^{(2)}_{m_2}( \mathfrak{B}_{v,t,h} ) - P( \mathfrak{B}_{v,t,h} ) \right| > \frac{\epsilon}{2} \right\}.
		\end{align*}
\endgroup

	   Let \(\mathcal{B} = \{ \mathfrak{B}_{v,t,h} : v, h \in \Sd, t \in \RR \}\). Each set \(\mathfrak{B}_{v,t,h}\) can be expressed as \(\{ x \in \RR^d : f_{v,t,h}(x) \geq 0 \}\), where
		\[
		f_{v,t,h}(x) = t - \langle R_v x, h \rangle = t - ( R_v^\top h )^\top x = t - w^\top x,
		\]
		with \(w = R_v^\top h \in \Sd\). Since \(R_v\) is an orthogonal matrix and \(h\) varies over \(\Sd\), \(w\) also varies over \(\Sd\).

		The class of functions \(\mathcal{F} = \{ f_{v,t,h} : \RR^d \to \RR \colon v, h \in \Sd, t \in \RR \} \) consists of affine functions of the form \(f_{w,t}(x) = t - w^\top x\), where \(w \in \Sd\) and \(t \in \RR\).  Thus \(\mathcal{B}\) is the set of half-spaces in \(\RR^d\), whose  Vapnik Chervonenkis (VC) dimension is \(d + 1\) \cite[Corollary~13.1]{Devroye_Lugosi1996}; the same holds for the class $\{\mathfrak B^0_{t,h} : h \in \Sd, t \in \RR\}$. Both suprema above are therefore suprema of empirical processes indexed by (subclasses of) the half-spaces of $\RR^d$, computed from the two independent i.i.d.\ samples $\aleph_1$ and $\aleph_2$. Note that $n \geq 4(d+2)$ implies $m_2 \geq m_1 = \lfloor n/2\rfloor \geq 2d + 4 > 2(d+1)$. By Theorems~12.8 and~13.3 in \cite{Devroye_Lugosi1996}, for $i \in \{1,2\}$ and all \(m_i > 2(d+1)\), we have
		\begin{equation*}
			\p \left\{ \sup_{v, h \in \Sd, t \in \RR } \left| P^{(i)}_{m_i}( \mathfrak{B}_{v,t,h} ) - P( \mathfrak{B}_{v,t,h} ) \right| > \epsilon \right\} \leq C ( m_i^{2(d+1)} + 1 ) e^{- 2 m_i \epsilon^2 },
		\end{equation*}
		where \(C \leq 4 e^8\), and an analogous bound holds for the class $\{\mathfrak B^0_{t,h}\}$. Applying these two inequalities with $\epsilon/2$ in place of $\epsilon$, and using $m_i \leq n$ together with $m_i \geq \lfloor n/2 \rfloor \geq (n-1)/2$, we obtain
		\begin{align*}
		\p \left\{ \|\widehat g_n-g\|_\infty > \epsilon \,\Big|\, \mathfrak H \right\} & \leq \sum_{i=1}^{2} C\bigl( m_i^{2(d+1)} + 1 \bigr) e^{-m_i \epsilon^2/2} \\
		& \leq 2C \bigl( n^{2d+2} + 1 \bigr) e^{-(n-1)\epsilon^2/4}.
		\end{align*}
		Since the right-hand side does not depend on $\mathfrak H$, the same bound holds unconditionally, which establishes \eqref{boundconvunif}.
		Hence
		\begin{align*}
			\p \left\{ \frac{\sqrt n}{a_n}\|\widehat g_n-g\|_\infty > \epsilon \right\}
			&\leq
			2 C ( n^{2d+2} + 1 ) e^{- (n-1) \epsilon^2 a_n^2/ (4n) }
			\\
			&\leq
			2 C ( n^{2d+2} + 1 )n^{-\epsilon^2 \delta_n/ 8 }
			\\
			&=
			2 C ( n^{2d+2-\epsilon^2 \delta_n/ 8} + n^{-\epsilon^2 \delta_n/ 8 }),
		\end{align*}
		where in the second inequality we used $a_n^2 = \delta_n \log n$ and $(n-1)/n \geq 1/2$ for $n \geq 2$.

		The assumption $\delta_n\to \infty$ implies that, for every $\epsilon >0$, there exists $N (=N(\epsilon))$ such that if $n\geq N$, then
		\[
		2 C ( n^{2d+2-\epsilon^2 \delta_n/ 8} + n^{-\epsilon^2 \delta_n/ 8 })
		< n^{-2}.
		\]

		The result follows from Borel-Cantelli's lemma \citep[Theorem~8.3.4]{Dudley2002}.
	\end{proof}

    \begin{lemma}   \label{lemma: continuity}
    Assume that \(P\) is an absolutely continuous distribution in $\RR^d$. Then for all $\mathfrak{H}$,  the function $g$ given by \eqref{eq: gH} is continuous; in particular, $\mathcal U^{\mathfrak H} = g^{-1}(\{0\})$ is a closed (hence compact) subset of $\Sd$.
    \end{lemma}

    \begin{proof}
    Consider $\mathfrak{H}$ fixed. As in~\eqref{eq: cont1} and~\eqref{eq: cont2}, for any $u_\ell \to u$ in $\Sd$ as $\ell \to \infty$ we have
        \begin{equation} \label{eq: cont3}
        \begin{aligned}
        \left\vert g(u_\ell) - g(u) \right\vert & \leq  \max_{j=1,\dots, k} \sup_{t \in \RR} \left| F_{\left\langle R_{u_\ell} X, h_j \right\rangle}(t) - F_{\left\langle R_{u} X, h_j \right\rangle}(t) \right|.
        \end{aligned}
        \end{equation}
    Since $R_{u_\ell} \to R_u$ as $\ell \to \infty$, we have $\langle R_{u_\ell} X, h_j\rangle \to \langle R_u X, h_j \rangle$ almost surely, and hence in distribution, for each $j = 1, \dots, k$. Because $P$ is absolutely continuous, it assigns mass zero to every hyperplane of $\RR^d$, so the distribution function of $\langle R_u X, h_j \rangle$ is continuous on $\RR$. By Pólya's theorem, weak convergence towards a limit with continuous distribution function implies uniform convergence of the distribution functions \citep[see, e.g.,][Lemma~2.11]{vanderVaart1998}. Hence the right-hand side of \eqref{eq: cont3} vanishes as $\ell \to \infty$, proving the continuity of $g$. The last claim follows since $g^{-1}(\{0\})$ is the preimage of a closed set under a continuous function defined on the compact set $\Sd$.
    \end{proof}

	For fixed $\mathfrak H$, the following theorem establishes almost sure Hausdorff convergence,
	\[
	d_H(\mathcal{U}^{\mathfrak H}, \mathcal{U}_n(\epsilon_n)) := \max\left\{ \sup_{a \in \mathcal{U}^{\mathfrak H}} d_E(a, \mathcal{U}_n(\epsilon_n)), \ \sup_{c \in \mathcal{U}_n(\epsilon_n)} d_E(c, \mathcal{U}^{\mathfrak H}) \right\} \to 0, 
	\]
    where $d_E(a,B)=\inf_{b\in B}\|a-b\|$. This finite semidistance is defined for nonempty subsets of $\Sd$, without requiring closedness; the empirical level set need not be closed. In part (i) of Theorem~\ref{teo:consistency}, it is eventually nonempty, whereas part (ii) treats the empty population set.

	\begin{theorem}\label{teo:consistency}
		Assume that \(P\) is an absolutely continuous distribution in $\RR^d$ with $\E(X) = 0_d$, and let $\mathfrak H = \{h_1, \dots, h_k\} \subset \Sd$ be a fixed set of directions (alternatively, a random set drawn once, independently of the sample, in which case the statements below hold almost surely conditionally on $\mathfrak H$). Let \(\beta_n\) be as in~\eqref{eq:betan}, \(\epsilon_n = 1/\beta_n\), and \(\mathcal{U}_n(\epsilon_n)\) as in \eqref{eq:estimator}. Then, almost surely:
		\begin{itemize}
		\item[(i)] if $\mathcal U^{\mathfrak H} \neq \emptyset$, then
		\(d_H(\mathcal{U}^{\mathfrak H}, \mathcal{U}_n(\epsilon_n)) \to 0 \quad \text{as } n \to \infty;
		\)
		\item[(ii)] if $\mathcal U^{\mathfrak H} = \emptyset$, then $\mathcal U_n(\epsilon_n) = \emptyset$ for all $n$ sufficiently large.
		\end{itemize}
	\end{theorem}

	\begin{proof}
		We work on a common probability space $(\Omega,\mathcal A,\p)$ and use the superscript $\omega$ to indicate a realization. From \Cref{prop:uniform_convergence} applied with $\delta_n = \log(n)$ (so that $\sqrt n / a_n = \beta_n$), there exists $\Omega_0\in \mathcal A$, such that, $\p(\Omega_0)=1$ and, for every $\omega \in \Omega_0$,
		\begin{equation*}
			\beta_n\|\widehat g_n^{\,\omega} - g\|_\infty \to 0, \quad \text{as } n \to \infty
		\end{equation*}
		(notice that $g$ and $\beta_n$ carry no superindex because they are not random once $\mathfrak H$ is fixed). Fix an element $\omega\in \Omega_0$. Then there exists $n_0^\omega \in \nat$ such that for all $n > n_0^\omega$,
		\begin{equation}\label{eq:twosided}
g(u)-\beta_n^{-1}<\widehat g_n^{\,\omega}(u)<g(u)+\beta_n^{-1},\qquad u\in\Sd.
\end{equation}

		For (ii), suppose $\mathcal U^{\mathfrak H} = \emptyset$. By Lemma~\ref{lemma: continuity}, $g$ is continuous and strictly positive on the compact set $\Sd$, so $c := \min_{u \in \Sd} g(u) > 0$. Take $n$ so large that $n > n_0^\omega$ and $2/\beta_n < c$. Then, by the lower bound in \eqref{eq:twosided}, for every $u \in \Sd$,
		\[
		\widehat g_n^{\,\omega}(u) > g(u) - \frac{1}{\beta_n} \geq c - \frac{1}{\beta_n} > \frac{1}{\beta_n} = \epsilon_n ,
		\]
		so that $\mathcal U_n^{\omega}(\epsilon_n) = \emptyset$.

		For (i), suppose $\mathcal U^{\mathfrak H} \neq \emptyset$. If $u \in \mathcal U^{\mathfrak H}$, then $g(u) = 0$ and the upper bound in \eqref{eq:twosided} gives $\widehat g_n^{\,\omega}(u) < 1/\beta_n = \epsilon_n$ for all $n > n_0^\omega$; consequently $\mathcal U^{\mathfrak H} \subseteq \mathcal U_n^\omega(\epsilon_n)$, and thus $\sup_{a \in \mathcal U^{\mathfrak H}} d_E(a, \mathcal U^\omega_n(\epsilon_n)) = 0$, for all $n > n_0^\omega$.

		Additionally, if $u \in \mathcal U_n^\omega(\epsilon_n)$ and $n > n_0^\omega$, the lower bound in \eqref{eq:twosided} gives $g(u) < \widehat g_n^{\,\omega}(u) + 1/\beta_n < 2/\beta_n$, so that $\mathcal{U}_n^\omega(\epsilon_n)\subseteq g^{-1}([0, 2/\beta_n))$. Therefore, it suffices to prove that
		\[
		\sup_{c \,\in\, g^{-1}([0, 2/\beta_n))} d_E\bigl(c, \mathcal U^{\mathfrak H}\bigr) \to 0, \quad \text{as } n \to \infty ;
		\]
		notice that no superindex appears here because none of the involved terms is random. Proceeding by contradiction, assume that there exist $\kappa>0$, indices $n_j\to\infty$, and points $u_j\in\Sd$ with $u_j\in g^{-1}([0,2/\beta_{n_j}))$ and $d_E(u_j,\mathcal U^{\mathfrak H})>\kappa$ for every $j\in\nat$. By compactness of $\Sd$, a subsequence satisfies $u_{j_\ell}\to u_0\in\Sd$. Since $0\leq g(u_{j_\ell})<2/\beta_{n_{j_\ell}}\to0$, continuity (Lemma~\ref{lemma: continuity}) gives $g(u_0)=0$, hence $u_0\in\mathcal U^{\mathfrak H}$. This contradicts $d_E(u_0,\mathcal U^{\mathfrak H})\geq\kappa$.
	\end{proof}

\begin{corollary}\label{cor:consistency-d2}
Let $d = 2$. Assume that $P$ is an absolutely continuous distribution on $\RR^2$ that satisfies the Carleman condition, with $\E(X) = 0_2$. Let $k \geq 2$ and let $\mathfrak H = \{h_1, \dots, h_k\}$ be drawn from $\nu^k$, independently of the sample, once and kept fixed as $n \to \infty$. Then, almost surely with respect to the joint distribution of $\mathfrak H$ and the sample:
\begin{itemize}
\item[(i)] if $\mathcal U \neq \emptyset$, then $d_H(\mathcal U, \mathcal U_n(\epsilon_n)) \to 0$ as $n \to \infty$;
\item[(ii)] if $\mathcal U = \emptyset$, then $\mathcal U_n(\epsilon_n) = \emptyset$ for all $n$ sufficiently large.
\end{itemize}
\end{corollary}

\begin{proof}
Assume first that $X$ is spherically symmetric. Then $X \eqd R_u X$ for every $u \in \Sd[1]$ (each $R_u$ being orthogonal), so $\mathcal U = \Sd[1]$ and $g_h(u) = 0$ for all $u, h \in \Sd[1]$, whence $\mathcal U^{\mathfrak H} = \Sd[1] = \mathcal U$ for every realization of $\mathfrak H$; claim (i) then follows from Theorem~\ref{teo:consistency}(i).

Assume now that $X$ is not spherically symmetric. By Theorem~\ref{theorem: dimension 2} applied to the pair $(h_1, h_2)$, with probability one with respect to $\mathfrak H$,
\[
\bigl\{ u \in \Sd[1] : g_{h_1}(u) = g_{h_2}(u) = 0 \bigr\} = \mathcal U .
\]
Since $\mathcal U \subseteq \mathcal U^{\mathfrak H} \subseteq \{u : g_{h_1}(u) = g_{h_2}(u) = 0\}$ always holds (see Section~\ref{sec: setup}), we conclude that $\mathcal U^{\mathfrak H} = \mathcal U$ almost surely. Both conclusions of Theorem~\ref{teo:consistency} were derived on the event $\Omega_0$ constructed in the proof of Theorem~\ref{teo:consistency}, which has probability one jointly in $(\mathfrak H, \aleph_n)$; intersecting $\Omega_0$ with the event $\{\mathcal U^{\mathfrak H} = \mathcal U\}$ yields an event of probability one on which the conclusions (i) and (ii) transfer verbatim from $\mathcal U^{\mathfrak H}$ to $\mathcal U$.
\end{proof}

\begin{remark}[Choice of the constant]\label{rem:calibration}
			For moderate $n$, the threshold $\epsilon_n=\log n/\sqrt n$ is too large (see \Cref{sec:simulation}), and the rate given in Theorem~\ref{teo:consistency} alone does not determine a constant. A constant can be chosen as follows, for a known center, $P$ absolutely continuous and, for simplicity, $n$ even: If $u_0$ is a symmetry axis, then for each $h\in\mathfrak H$ the projections $\langle X_i,h\rangle$, $X_i\in\aleph_1$, and $\langle R_{u_0}X_i,h\rangle$, $X_i\in\aleph_2$, are two independent samples of size $n/2$ from the same continuous distribution $F_h$. Hence $(\sqrt n/2)\,\widehat g_{n,h}(u_0)$ is the normalized two-sample Kolmogorov--Smirnov statistic under the null hypothesis: its distribution, that of a variable $K_n$ say, does not depend on $P$ or $h$, and $K_n$ converges in distribution to a variable $K$ with the Kolmogorov distribution \citep{Smirnov1939}. Since $x\mapsto(x-t)_+=\max\{x-t,0\}$ is convex, $\bigl((\sqrt n/2)\,\widehat g_n(u_0)-t\bigr)_+\leq k^{-1}\sum_{j=1}^k\bigl((\sqrt n/2)\,\widehat g_{n,h_j}(u_0)-t\bigr)_+$, whatever the dependence among the directions, and Markov's inequality yields, for $0\leq t<c$,
			\begin{equation}\label{eq:markov}
				\pr\left(\frac{\sqrt n}2 \widehat g_n(u_0)\geq c\right)\leq\frac{\E(K_n-t)_+}{c-t}.
			\end{equation}
			The family $\{K_n\}_n$ is uniformly integrable. Indeed, if $F^{(r)}$ denotes the empirical distribution function of the $r$-th projected sample, the triangle inequality gives $(\sqrt n/2)\,\widehat g_{n,h}(u_0)\leq2^{-1/2}\bigl(\sqrt{n/2}\,\|F^{(1)}-F_h\|_\infty+\sqrt{n/2}\,\|F^{(2)}-F_h\|_\infty\bigr)$, and the Dvoretzky--Kiefer--Wolfowitz inequality \cite{Massart1990}, $\pr\bigl(\sqrt{n/2}\,\|F^{(r)}-F_h\|_\infty>x/\sqrt2\bigr)\leq2e^{-x^2}$, gives $\pr(K_n>x)\leq4e^{-x^2}$ for every $x>0$ and every $n$; in particular $\sup_n\E K_n^2\leq4$. Convergence in distribution and uniform integrability imply $\E(K_n-t)_+\to\E(K-t)_+$, so letting $n\to\infty$ in~\eqref{eq:markov},
			\[
			\limsup_{n\to\infty}\pr\left(\frac{\sqrt n}2 \widehat g_n(u_0)\geq c\right)\leq\inf_{0\leq t<c}\frac{\E(K-t)_+}{c-t}.
			\]
			By numerical integration, the right-hand side is at most $0.05$ for $c=1.524$. We therefore use the threshold $\epsilon_n^*:=3.048/\sqrt n$, which retains each symmetry axis in $\mathcal U_n(\epsilon_n^*)$ with asymptotic probability at least $0.95$, whatever $P$, $k$ and $\mathfrak H$. This choice is conservative, since the bound disregards the averaging over the $k$ directions. Moreover, since $\epsilon_n^*\to0$, the argument of Theorem~\ref{teo:consistency} still shows that, almost surely, $\mathcal U_n(\epsilon_n^*)=\emptyset$ for all $n$ sufficiently large if $\mathcal U^{\mathfrak H}=\emptyset$, and $\sup_{v\in\mathcal U_n(\epsilon_n^*)}d_E(v,\mathcal U^{\mathfrak H})\to0$ otherwise.
		\end{remark}

\begin{remark}
For symmetry about a prescribed axis $u_0$, the statistic
$\widehat g_n(u_0)$ suggests a random-projection test;
see \citet{fraiman2017some,fraiman2024application} for related
invariance tests. For the hypothesis that some symmetry axis exists,
a natural statistic is $T_n:=\min_{u\in\Sd}\widehat g_n(u)$.
Under the assumptions of Corollary~\ref{cor:consistency-d2},
$|T_n-\min_u g(u)|\leq\|\widehat g_n-g\|_\infty$ implies
$T_n\to0$ almost surely if $\mathcal U\neq\emptyset$, whereas
$T_n\to c_{\mathfrak H}:=\min_u g(u)>0$ almost surely if
$\mathcal U=\emptyset$, by identification (Corollary~\ref{cor:consistency-d2}), continuity (Lemma~\ref{lemma: continuity}) and compactness.
Thus, the statistic $T_n$ separates the hypotheses asymptotically; calibration at a prescribed significance level requires a separate
analysis and is beyond the scope of this paper.
\end{remark}

\begin{remark}\label{rem:center}
Put $\mu=\E X$ and $e_n=\bar X_n-\mu$, where
$\bar X_n=n^{-1}\sum_{i=1}^nX_i$. Denote by $\widetilde g_n$ the estimator
obtained by replacing $X_i$ with $X_i-\bar X_n$, and let $g$
be the population criterion for $X-\mu$. In addition to
absolute continuity of $P$, assume
(\textsf{A}$_1$): $\E\|X\|^2<\infty$ (automatic under Carleman's
condition) and 
\begin{equation}\label{eq:assD}
\mbox{(\textsf{A}$_2$):} \quad L:=\sup_{w\in\Sd}
\operatorname*{ess\,sup}_{t\in\RR}
f_{\langle X,w\rangle}(t)<\infty,
\end{equation}
where $f_{\langle X,w\rangle}$ is the density of $\langle X, u \rangle$.

For $r=1,2$, let $F^{(r)}_{n,w}$ and
$\widetilde F^{(r)}_{n,w}$ be the empirical distribution
functions of $\langle X_i-\mu,w\rangle$ and
$\langle X_i-\bar X_n,w\rangle$, respectively, in subsample
$\aleph_r$. They satisfy the identity
\(\widetilde F^{(r)}_{n,w}(t)
=F^{(r)}_{n,w}\bigl(t+\langle e_n,w\rangle\bigr),\) \(|\langle e_n,w\rangle|\leq\|e_n\|.
\)
Write $F_w=F_{\langle X-\mu,w\rangle}$ and
$A_{n,r}=\sup_{w\in\Sd,\,t\in\RR}
|F^{(r)}_{n,w}(t)-F_w(t)|$.
By \eqref{eq:assD}, every $F_w$ is $L$-Lipschitz, so the
triangle inequality gives
$\sup_t|\widetilde F^{(r)}_{n,w}(t)-F_w(t)|
\leq A_{n,r}+L\|e_n\|$.
Applying this with $w=h$ and $w=R_vh$ to the two empirical
distributions defining the criterion yields
\[
\|\widetilde g_n-g\|_\infty
\leq A_{n,1}+A_{n,2}+2L\|e_n\|.
\]
This bound holds sample by sample, so no independence between
the estimated center and the empirical processes is required.

Both $A_{n,r}$ are empirical-process suprema over half-spaces.
The exponential bounds in the proof of
\Cref{prop:uniform_convergence} and Borel--Cantelli therefore
give $A_{n,1}+A_{n,2}
=O_{\mathrm{a.s.}}(\sqrt{\log n/n})$.
The coordinatewise law of the iterated logarithm
\citep{HartmanWintner1941} gives
$\|e_n\|=O_{\mathrm{a.s.}}(\sqrt{\log\log n/n})$.
Consequently,
$(\sqrt n/a_n)\|\widetilde g_n-g\|_\infty\to0$ almost surely
for every $a_n=\sqrt{\delta_n\log n}$ with $\delta_n\to\infty$.

Thus \Cref{prop:uniform_convergence} remains valid for
$\widetilde g_n$. Under their remaining assumptions,
Theorem~\ref{teo:consistency} and
Corollary~\ref{cor:consistency-d2} also remain valid when the
level sets are formed from $\widetilde g_n$, with the same
threshold $\epsilon_n=\log n/\sqrt n$.
If the center is known, deterministic centering requires
no additional assumptions.
\end{remark}

\section{Simulation study}\label{sec:simulation}

We report simulations for two planar distributions with known axes, varying the sample size $n$ and the number of projection directions $k$. We also describe a third scenario with no symmetry axis.

\noindent \textbf{Scenario~1.}
A bivariate normal distribution with unit variances and a correlation coefficient of $\rho = 0.7$.  A sample drawn from this distribution of size $n=1000$ is depicted in \Cref{fig: Gauss}.

\noindent \textbf{Scenario~2.}
A uniform distribution on the square $[-1, 1]^2$. A sample  consisting of $n=10000$ data points is depicted in \Cref{fig: square}.

	\noindent \textbf{Scenario~3.} A uniform distribution on the triangle with vertices $(0,0)\tr$, $(2,0)\tr$ and $(-1/2,1)\tr$, translated so that its centroid lies at the origin. Its side lengths are $2$, $\sqrt5/2$ and $\sqrt{29}/2$, so the triangle is scalene and the distribution has no axis of symmetry. A sample of size $n=1000$ is depicted in \Cref{fig: triangle}. 
    
    In Scenario~3, the bounded support ensures Carleman's condition, so Theorem~\ref{theorem: dimension 2} gives $\mathcal U^{\mathfrak H}=\mathcal U=\emptyset$ almost surely. By Corollary~\ref{cor:consistency-d2}(ii), $\mathcal U_n(\epsilon_n)$ is eventually empty almost surely. Continuity of $g$ (Lemma~\ref{lemma: continuity}) and compactness of $\mathbb S^1$ give $c:=\min_{u\in\mathbb S^1}g(u)>0$; hence \Cref{prop:uniform_convergence} yields $\min_u\widehat g_n(u)\to c>0$ almost surely. In Scenarios~1 and~2, this minimum converges to zero.

In finite samples, however, $\epsilon_n=\log n/\sqrt n$ exceeds the average of $\min_u\widehat g_n(u)$ in Scenario~3 for all sample sizes considered (\Cref{tab:min_gn}), so $\mathcal U_n(\epsilon_n)$ is nonempty in most replications. We therefore also use the threshold $\epsilon_n^*=3.048/\sqrt n$ of Remark~\ref{rem:calibration}.

\begin{figure}[!htbp]%
	\centering
	\begin{subfigure}[t]{.32\linewidth}
		\includegraphics[width=\linewidth]{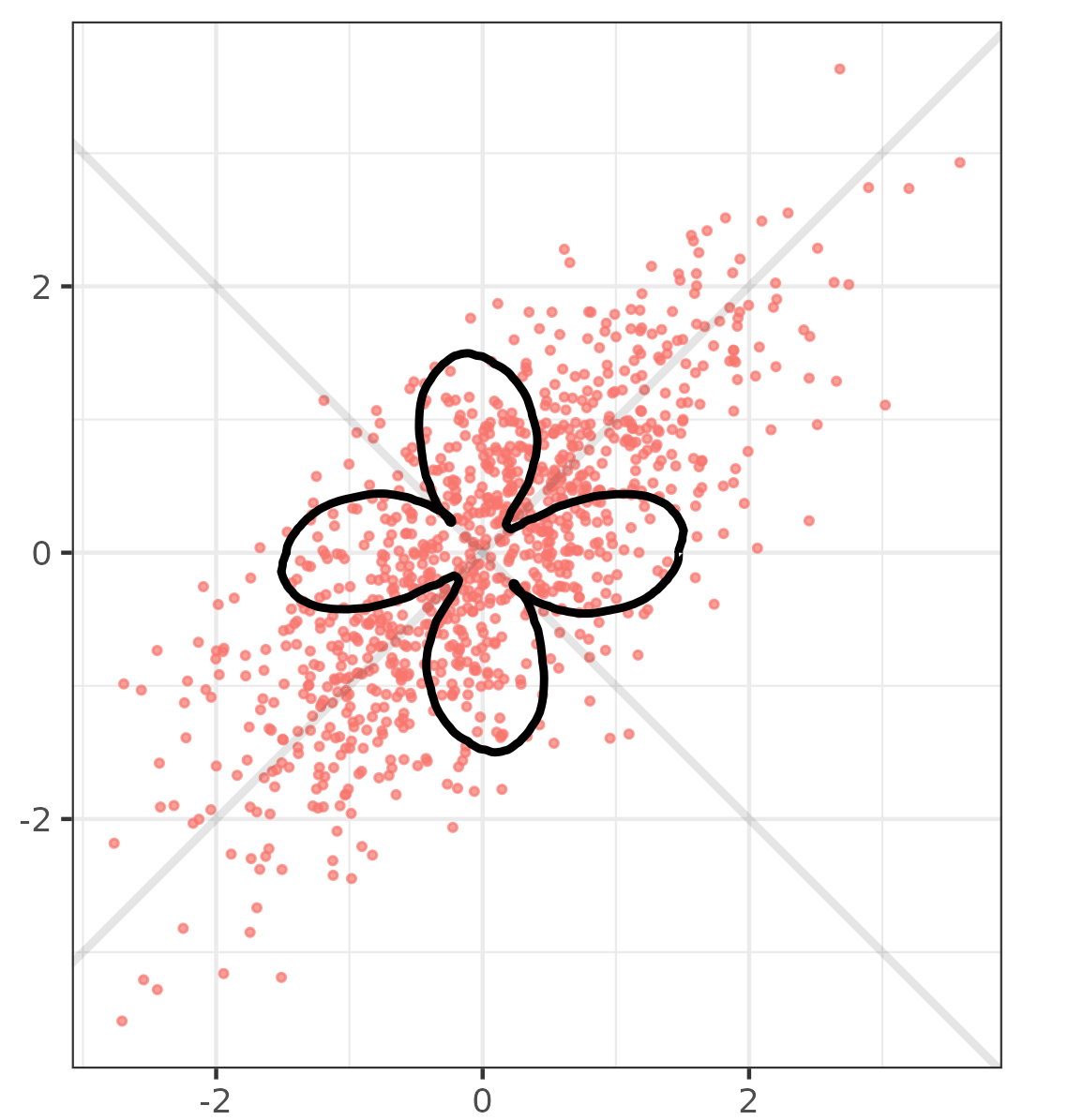}
		\subcaption{} \label{fig: Gauss}
	\end{subfigure}\hfill
	\begin{subfigure}[t]{.32\linewidth}
		\includegraphics[width=\linewidth]{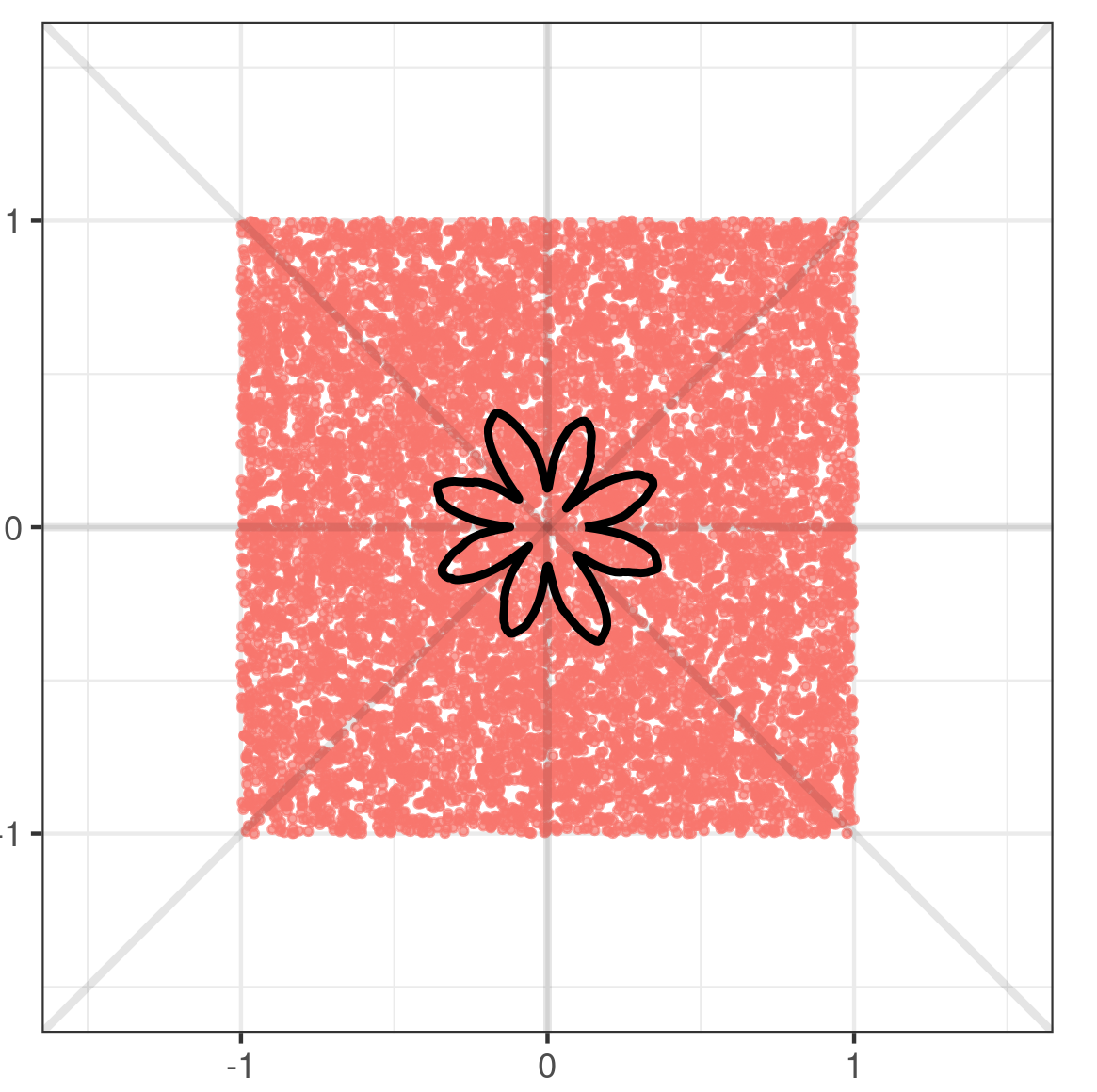}
		\subcaption{} \label{fig: square}
	\end{subfigure}\hfill
	\begin{subfigure}[t]{.32\linewidth}
		\includegraphics[width=\linewidth]{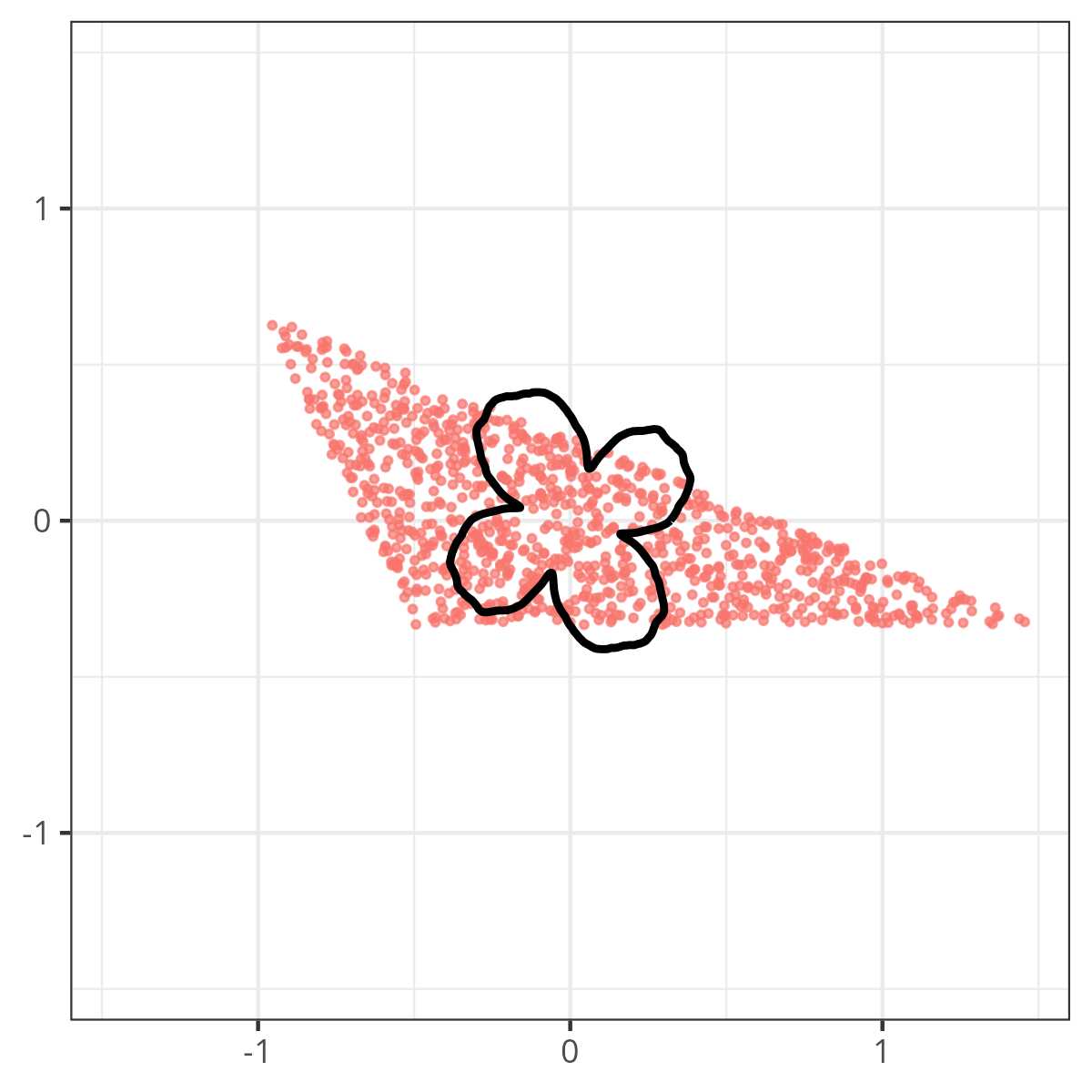}
		\subcaption{} \label{fig: triangle}
	\end{subfigure}
	\caption{
		In each subfigure, simulated data are depicted as red points; the black line illustrates the estimate of the function $g$ in a polar plot format, and the gray lines indicate the true axes of symmetry. (a): a bivariate normal distribution (Scenario~1), (b): a uniform distribution on a square (Scenario~2), and (c): a uniform distribution on a scalene triangle (Scenario~3), with no axis of symmetry.
	}
    \label{fig: scenarios}
	
\end{figure}

In Figure~\ref{fig: scenarios}, for $u=(\cos\theta,\sin\theta)^\top$, the black curve has polar coordinates $(\widehat g_n(u),\theta)$. Directions where it approaches the origin are candidate axes of symmetry. The true axes appear as gray lines.

For each of Scenarios~1--3, we simulated 500 independent datasets. We evaluated $\widehat g_n$ on an equally spaced grid $\{u_j\}_{j=1}^{200}\subset\Sd[1]$ and selected local minima of the resulting discrete profile. The projection directions $h_1,\dots,h_k$ were drawn independently and uniformly on $\Sd[1]$. Since $R_{-u}=R_u$, all our numerical summaries count axes modulo sign, meaning that throughout this section, we restrict analysis to $u$ in a semi-circle of $\Sd[1]$.
The grid is $u_j=(\cos\theta_j,\sin\theta_j)\tr$ with $\theta_j$ equally spaced in $[0,\pi)$, and all angular errors are measured with the axis distance $d(u,v)=\arccos|u\tr v|$, which identifies $u$ with $-u$ and is therefore insensitive to the wrap-around at $\theta=0$ and $\theta=\pi$.

We used the AMPD algorithm \cite{a5040588} to select representative local minima, and evaluated their number and locations empirically.\footnote{This implementation is restricted to dimension two. For $d>2$, optimizing the noisy, non-convex, discrete criterion on the sphere requires a different numerical procedure, outside the scope of this study.} The AMPD is a specialized algorithm for automatic detection of peaks in noisy periodic signals, which we use to identify the number and locations of true (non-spurious) local minima of our objective function.

	A connection of AMPD with the level-set estimator follows from uniform convergence.
	Under the assumptions of Corollary~\ref{cor:consistency-d2}, when
	$\mathcal U\neq\emptyset$, almost surely every sequence of directions $u_n$
	satisfying $\widehat g_n(u_n)\to0$ also satisfies $d_E(u_n,\mathcal U)\to0$.
	Indeed, $0\leq g(u_n)\leq\widehat g_n(u_n)+\|\widehat g_n-g\|_\infty\to0$;
	continuity and compactness then imply that $u_n$ approaches the zero set
	$\mathcal U$. In particular, this applies to selected candidates satisfying
	$\widehat g_n(u_n)<\epsilon_n$. If $\mathcal U=\emptyset$, no candidate
	satisfies this inequality for all sufficiently large $n$.
	These observations justify interpreting sufficiently low minima as
	approximate symmetry directions, but, of course, do not establish that AMPD detects
	every axis or consistently estimates their number. The simulations assess
	the number and locations of the selected minima.

\Cref{tab:todas_frecuencias} reports the relative frequencies of the numbers of local minima detected by AMPD in Scenario~2, for $n=200,500,1000,2000,5000$ and $k=200$. The entries are proportions of 500 replications; the correct number of true axes of the square is four. The last column records the replications in which the AMPD routine terminated with an error and returned no candidate; these are counted as incorrect detections throughout, and they vanish for $n\geq1000$. We examine other values of $k$ below.
For Scenario~1, our algorithm correctly detected the two local minima in 100\% of the replications for all sample sizes greater than 200. For $n=200$, it detected an incorrect number of local minima in only 2 out of 500 replications.

\begin{table}[!htbp]
	\centering
	\small
	\begin{tabular}{@{}rrrrrrrr@{}}
		\toprule
		$n$ & 1 & 2 & 3 & 4 & 5 & $\geq 6$ & fail \\
		\midrule
		200 & 0.142 & 0.468 & 0.148 & \textbf{0.176} & 0.018 & 0.016 & 0.032 \\
		500 & 0.112 & 0.232 & 0.126 & \textbf{0.498} & 0.014 & 0.014 & 0.004 \\
		1000 & 0.034 & 0.136 & 0.034 & \textbf{0.790} & 0.004 & 0.002 & 0.000 \\
		2000 & 0.002 & 0.024 & 0.004 & \textbf{0.970} & 0.000 & 0.000 & 0.000 \\
		5000 & 0.000 & 0.000 & 0.000 & \textbf{1.000} & 0.000 & 0.000 & 0.000 \\
		\bottomrule
	\end{tabular}
	\caption{Relative frequencies of the numbers of local minima detected by AMPD in Scenario~2, with $k=200$ projection directions. Column headings give the number of detected minima; \textit{fail} is the proportion of replications in which the AMPD routine returned an error. Each entry is a proportion of 500 replications and each row sums to one. The correct number is four (bold column).}
	\label{tab:todas_frecuencias}\label{tab:sub1}\label{tab:sub2}\label{tab:sub3}\label{tab:sub4}\label{tab:sub5}
\end{table}

For the error plots, among replications with the correct number of minima, we matched each estimated axis to its closest true axis and averaged the angular distances. This directed nearest-axis error does not require a bijective matching. \Cref{fig:errors} shows the errors for Scenarios~1 and~2 with $k=200$. \Cref{tab:error_fixed_k} gives the mean and median angular errors from the sample-size experiment.

\begin{table}[!htbp]
	\centering
	\small
	\begin{tabular}{@{}rcc@{}}
		\toprule
		$n$ & Scenario~1 & Scenario~2 \\
		\midrule
		200 & 0.054\,(0.047) & 0.094\,(0.094) \\
		500 & 0.030\,(0.031) & 0.051\,(0.043) \\
		1000 & 0.021\,(0.016) & 0.036\,(0.031) \\
		2000 & 0.013\,(0.016) & 0.022\,(0.020) \\
		5000 & 0.008\,(0.008) & 0.012\,(0.012) \\
		\bottomrule
	\end{tabular}
	\caption{Angular-error summaries (in radians) for the sample-size experiment in Scenarios~1 and~2, with $k=200$. Entries are mean errors, with medians in parentheses, computed over the replications in which AMPD detected the correct number of minima; the number of such replications is reported in Table~\ref{tab:minima_n_k}. The error of an estimated axis is its axial distance $\arccos|u^\top v|$ to the nearest true axis.}
	\label{tab:error_fixed_k}\label{tab:error_normal}\label{tab:error_unif}
\end{table}

\begin{figure}[!htbp]
	\centering
	\includegraphics[width = \linewidth]{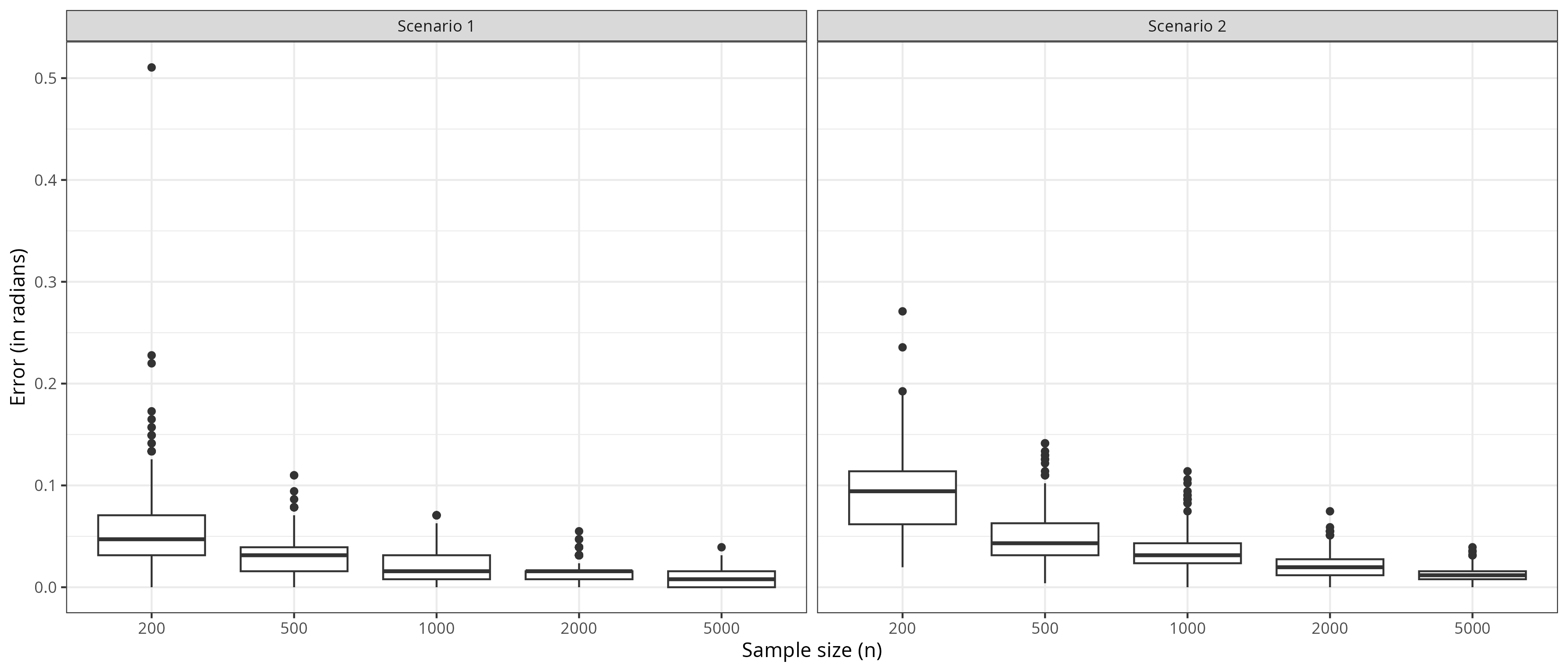}
	\caption{Average angular error for different sample sizes for Scenario~1 and 2, for $k=200$ projections.}
	\label{fig:errors}
\end{figure}

We next used $k=2,10,50,100,150,200$ projections. \Cref{tab:minima_n_k} reports the proportions of replications with the correct number of minima in Scenarios~1 and~2. \Cref{tab:error_n_k} gives the corresponding mean and median angular-error summaries for the study varying $n$ and $k$; \Cref{fig1:error_k,fig2:error_k} display the angular errors conditional on detecting the correct number of minima.

\begin{table}[!htbp]
	\centering
	\small
	\begin{tabular}{@{}rrrrrrr@{}}
		\toprule
		$n$ & $k=2$ & $k=10$ & $k=50$ & $k=100$ & $k=150$ & $k=200$ \\
		\midrule
		\multicolumn{7}{@{}l}{Scenario~1 (two axes)} \\
		200 & 0.864 & 0.960 & 0.996 & 0.998 & 0.998 & 0.996 \\
		500 & 0.900 & 0.992 & 1.000 & 1.000 & 1.000 & 1.000 \\
		1000 & 0.892 & 0.986 & 1.000 & 1.000 & 1.000 & 1.000 \\
		2000 & 0.888 & 0.992 & 1.000 & 1.000 & 1.000 & 1.000 \\
		5000 & 0.896 & 0.992 & 1.000 & 1.000 & 1.000 & 1.000 \\
		\midrule
		\multicolumn{7}{@{}l}{Scenario~2 (four axes)} \\
		200 & 0.140 & 0.146 & 0.176 & 0.160 & 0.182 & 0.176 \\
		500 & 0.318 & 0.406 & 0.472 & 0.482 & 0.486 & 0.498 \\
		1000 & 0.456 & 0.646 & 0.764 & 0.786 & 0.788 & 0.790 \\
		2000 & 0.656 & 0.878 & 0.966 & 0.968 & 0.964 & 0.970 \\
		5000 & 0.872 & 0.988 & 0.998 & 0.998 & 1.000 & 1.000 \\
		\bottomrule
	\end{tabular}
	\caption{Proportions of 500 replications in which AMPD detected the correct number of minima: two for Scenario~1 and four for Scenario~2. Replications in which the AMPD routine returned an error count as incorrect. The case $k=200$ of Scenario~2 corresponds to the bold column of Table~\ref{tab:todas_frecuencias}.}
	\label{tab:minima_n_k}\label{tab:minima_gauss_k_n}\label{tab:minima_unif_k_n}
\end{table}

\begin{table}[!htbp]
	\centering
	\footnotesize
	\setlength{\tabcolsep}{2.5pt}
	\begin{tabular}{@{}rcccccc@{}}
		\toprule
		$n$ & $k=2$ & $k=10$ & $k=50$ & $k=100$ & $k=150$ & $k=200$ \\
		\midrule
		\multicolumn{7}{@{}l}{Scenario~1} \\
		200 & 0.139\,(0.102) & 0.074\,(0.055) & 0.061\,(0.047) & 0.057\,(0.047) & 0.056\,(0.047) & 0.054\,(0.047) \\
		500 & 0.081\,(0.055) & 0.039\,(0.031) & 0.031\,(0.031) & 0.031\,(0.031) & 0.030\,(0.024) & 0.030\,(0.031) \\
		1000 & 0.055\,(0.039) & 0.025\,(0.024) & 0.021\,(0.016) & 0.021\,(0.016) & 0.021\,(0.016) & 0.021\,(0.016) \\
		2000 & 0.037\,(0.024) & 0.017\,(0.016) & 0.014\,(0.016) & 0.013\,(0.016) & 0.013\,(0.008) & 0.013\,(0.016) \\
		5000 & 0.023\,(0.016) & 0.011\,(0.008) & 0.009\,(0.008) & 0.009\,(0.008) & 0.008\,(0.008) & 0.008\,(0.008) \\
		\midrule
		\multicolumn{7}{@{}l}{Scenario~2} \\
		200 & 0.140\,(0.126) & 0.116\,(0.118) & 0.097\,(0.094) & 0.093\,(0.092) & 0.095\,(0.094) & 0.094\,(0.094) \\
		500 & 0.102\,(0.094) & 0.076\,(0.067) & 0.057\,(0.051) & 0.053\,(0.047) & 0.051\,(0.047) & 0.051\,(0.043) \\
		1000 & 0.088\,(0.075) & 0.053\,(0.047) & 0.040\,(0.035) & 0.037\,(0.035) & 0.036\,(0.031) & 0.036\,(0.031) \\
		2000 & 0.071\,(0.059) & 0.033\,(0.027) & 0.024\,(0.024) & 0.022\,(0.020) & 0.022\,(0.020) & 0.022\,(0.020) \\
		5000 & 0.048\,(0.035) & 0.018\,(0.016) & 0.014\,(0.012) & 0.013\,(0.012) & 0.012\,(0.012) & 0.012\,(0.012) \\
		\bottomrule
	\end{tabular}
	\caption{Angular-error summaries (in radians) for the experiment varying $n$ and $k$ in Scenarios~1 and~2. Entries are mean errors, with medians in parentheses, computed over the replications in which AMPD detected the correct number of minima. The $k=200$ column reproduces Table~\ref{tab:error_fixed_k}.}
	\label{tab:error_n_k}\label{tab:error_gauss_k}\label{tab:error_unif_k}
\end{table}

The angular errors decrease as the sample size increases, and for each $n$ they stabilize beyond $k\approx50$. The benefit of increasing $k$ is, however, strongly dependent on $n$: for $n=2000$ the detection rate in Scenario~2 rises from $0.656$ at $k=2$ to $0.966$ at $k=50$, whereas for $n=200$ it remains below $0.19$ for every $k$ considered. Additional projection directions cannot compensate for an insufficient sample size. The method performs better for the Gaussian data (Scenario~1) than for the uniform data (Scenario~2).

\begin{figure}[!htbp]%
	\centering
	\includegraphics[width=\linewidth]{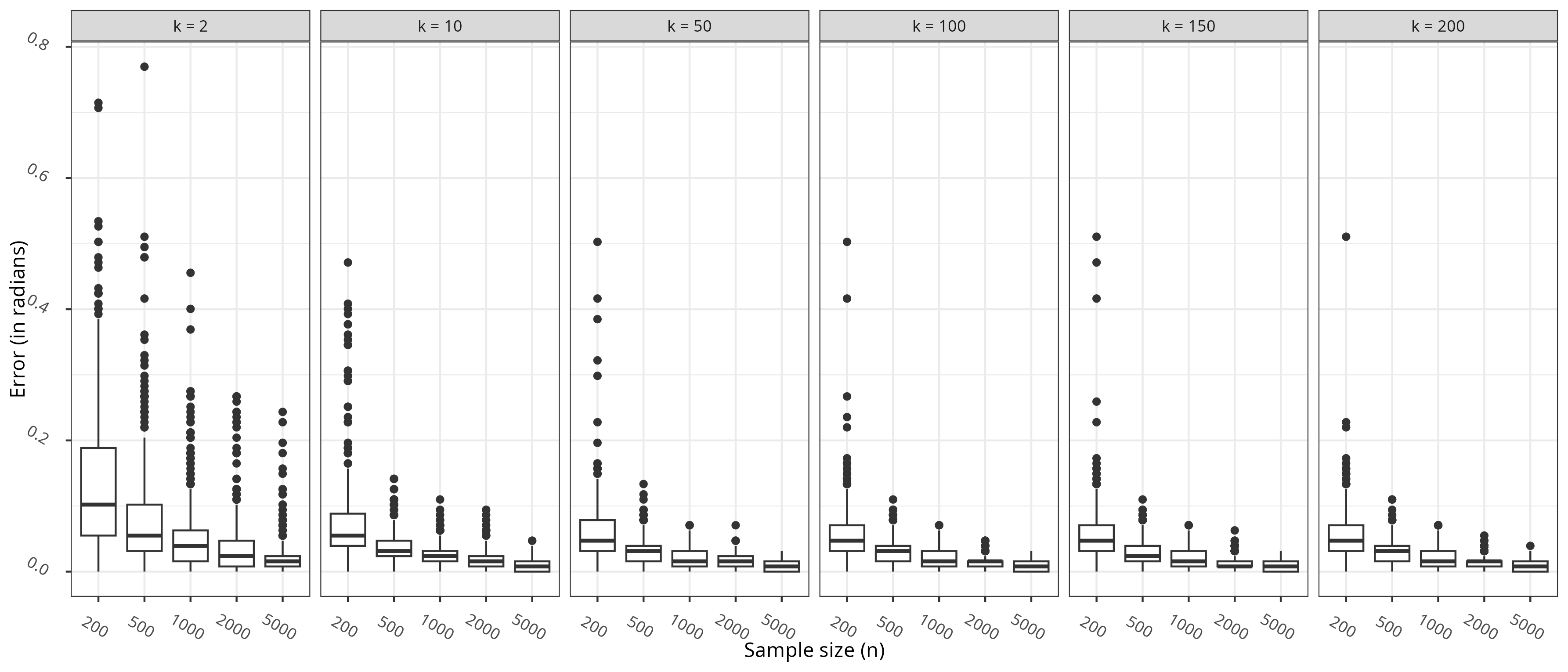}
	\caption{Average angular error when considering different sample sizes, and different values of $k$, for Scenario~1.}
	\label{fig1:error_k}
	
\end{figure}

\begin{figure}[!htbp]%
	\centering
	\includegraphics[width=\linewidth]{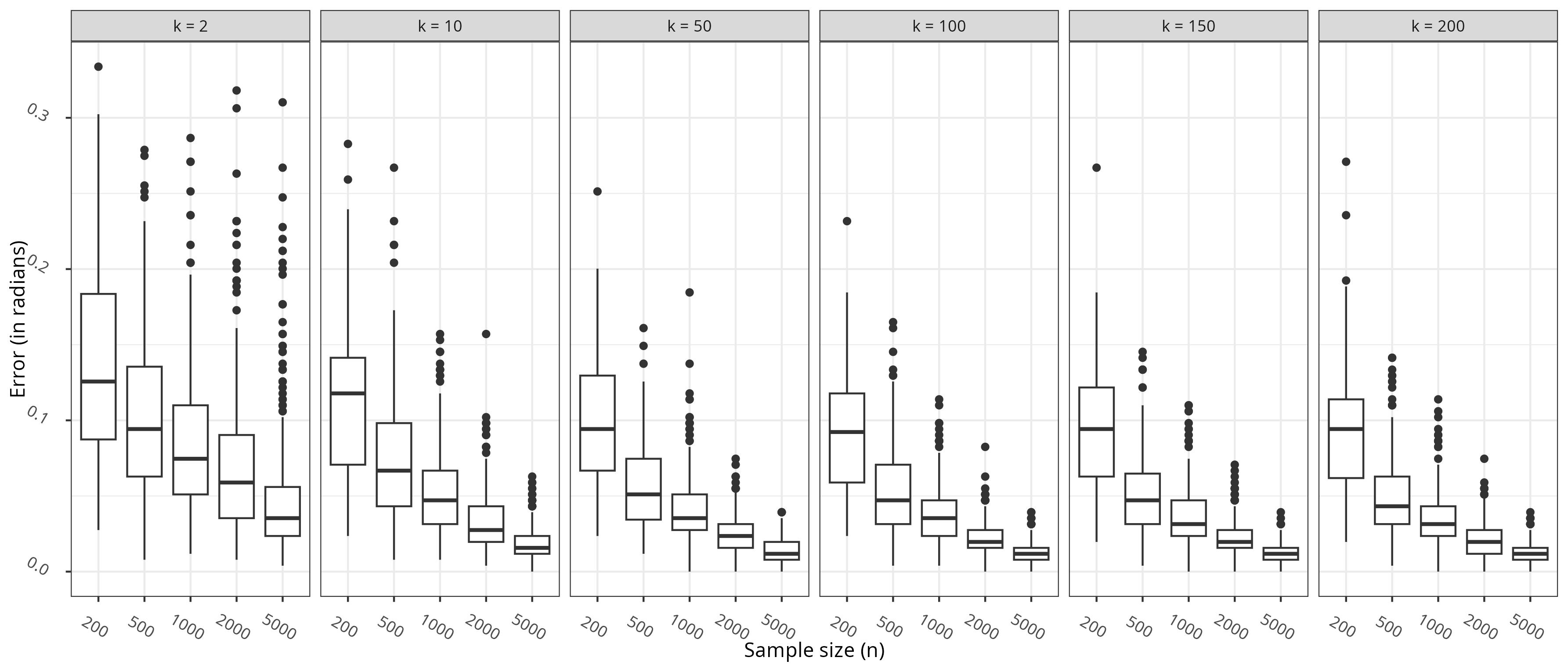}
	\caption{Average angular error when considering different sample sizes, and different values of $k$, for Scenario~2.}
	\label{fig2:error_k}
	
\end{figure}

For Scenario~3, \Cref{tab:min_gn} reports the mean and standard deviation of $\min_u\widehat g_n(u)$ over the 500 replications, with $k=200$, together with the corresponding values for Scenarios~1 and~2; \Cref{fig:min_gn} displays its distribution. In Scenarios~1 and~2 the minimum decreases towards zero as $n$ increases, whereas in Scenario~3 it remains bounded away from zero. 
The last column of \Cref{tab:min_gn} gives the proportion of replications of Scenario~3 in which $\mathcal U_n(\epsilon_n^*)$ contains no grid point, so that no axis is detected. In Scenarios~1 and~2 this proportion is $0$ for every sample size, in agreement with Remark~\ref{rem:calibration}. In Scenario~3 it is negligible for $n\leq1000$, reflecting the conservative choice of $\epsilon_n^*$, and rises to $0.892$ for $n=2000$ and to $1$ for $n=5000$.

\begin{table}[!htbp]
	\centering
	\small
	\setlength{\tabcolsep}{4pt}
	\begin{tabular}{@{}rcccccc@{}}
 \toprule
$n$ & $\epsilon_n$ & $\epsilon_n^*$ & Scenario~1 & Scenario~2 & Scenario~3 & $\emptyset$ (Sc.~3) \\ 
  \midrule
200 & 0.375 & 0.216 & 0.102 (0.017) & 0.093 (0.012) & 0.122 (0.017) & 0.000 \\ 
  500 & 0.278 & 0.136 & 0.066 (0.011) & 0.062 (0.008) & 0.095 (0.011) & 0.002 \\ 
  1000 & 0.218 & 0.096 & 0.047 (0.008) & 0.044 (0.006) & 0.083 (0.008) & 0.060 \\ 
  2000 & 0.170 & 0.068 & 0.033 (0.005) & 0.032 (0.004) & 0.075 (0.006) & 0.892 \\ 
  5000 & 0.120 & 0.043 & 0.021 (0.003) & 0.020 (0.003) & 0.069 (0.004) & 1.000 \\ 
   \bottomrule
	\end{tabular}
    \caption{Summaries of $T_n=\min_u\widehat g_n(u)$ in Scenarios~1--3, with $k=200$ projection directions. The first three columns give the sample size $n$, the threshold $\epsilon_n=\log n/\sqrt n$ and the threshold $\epsilon_n^*=3.048/\sqrt n$ of Remark~\ref{rem:calibration}; the next three columns give the mean of $T_n$, with the standard deviation in parentheses, over 500 replications. The last column ($\emptyset$ (Sc.~3)) gives the proportion of replications of Scenario~3 in which $T_n\geq\epsilon_n^*$, that is, in which $\mathcal U_n(\epsilon_n^*)$ contains no grid point.} The population value of $T_n$ for Scenario~3 is $c\approx0.062$; for Scenarios~1 and~2 it is zero.
	\label{tab:min_gn}
\end{table}

\begin{figure}[!htbp]
	\centering
	\includegraphics[width = \linewidth]{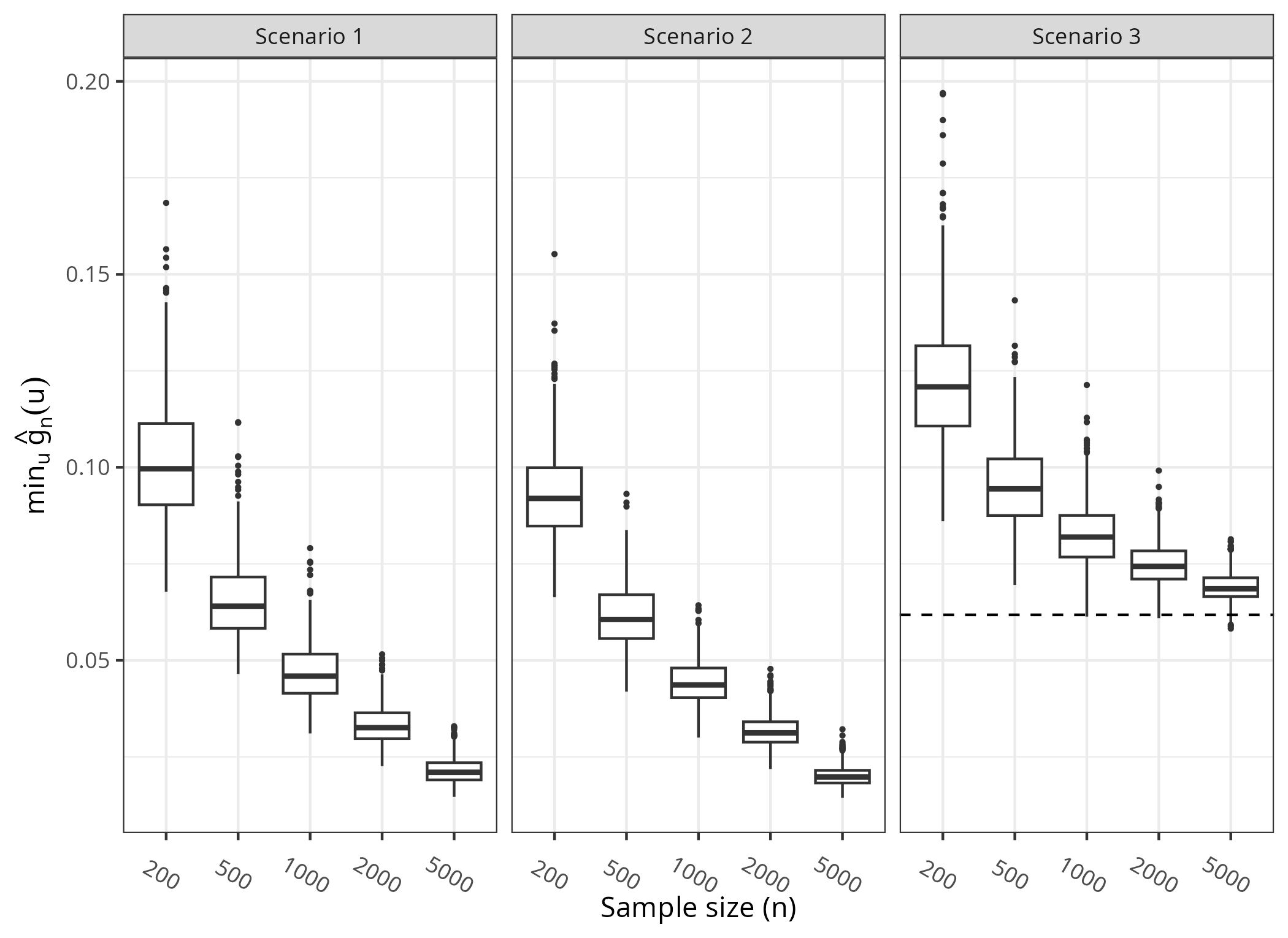}
	\caption{Distribution of $T_n=\min_u\widehat g_n(u)$ over 500 replications for different sample sizes in Scenarios~1--3, with $k=200$ projections. The dashed line in Scenario~3 panel is the population value $c\approx0.062$.}
	\label{fig:min_gn}
\end{figure}

		\section{Application to real data: analysis of medical images}\label{sec:aplicacion}

	We illustrate the method on a chest radiograph. The symmetry measure of \citet{hogeweg2017fast} compares image features at points related by reflection across an estimated axis, then averages the local comparisons to obtain a global measure. Our estimator offers an alternative way to determine that axis 
    from a point cloud sampled within the lung region.

    \begin{figure}[!htbp]%
		\centering
			\begin{subfigure}[t]{.48\linewidth}
				\includegraphics[width=\linewidth]{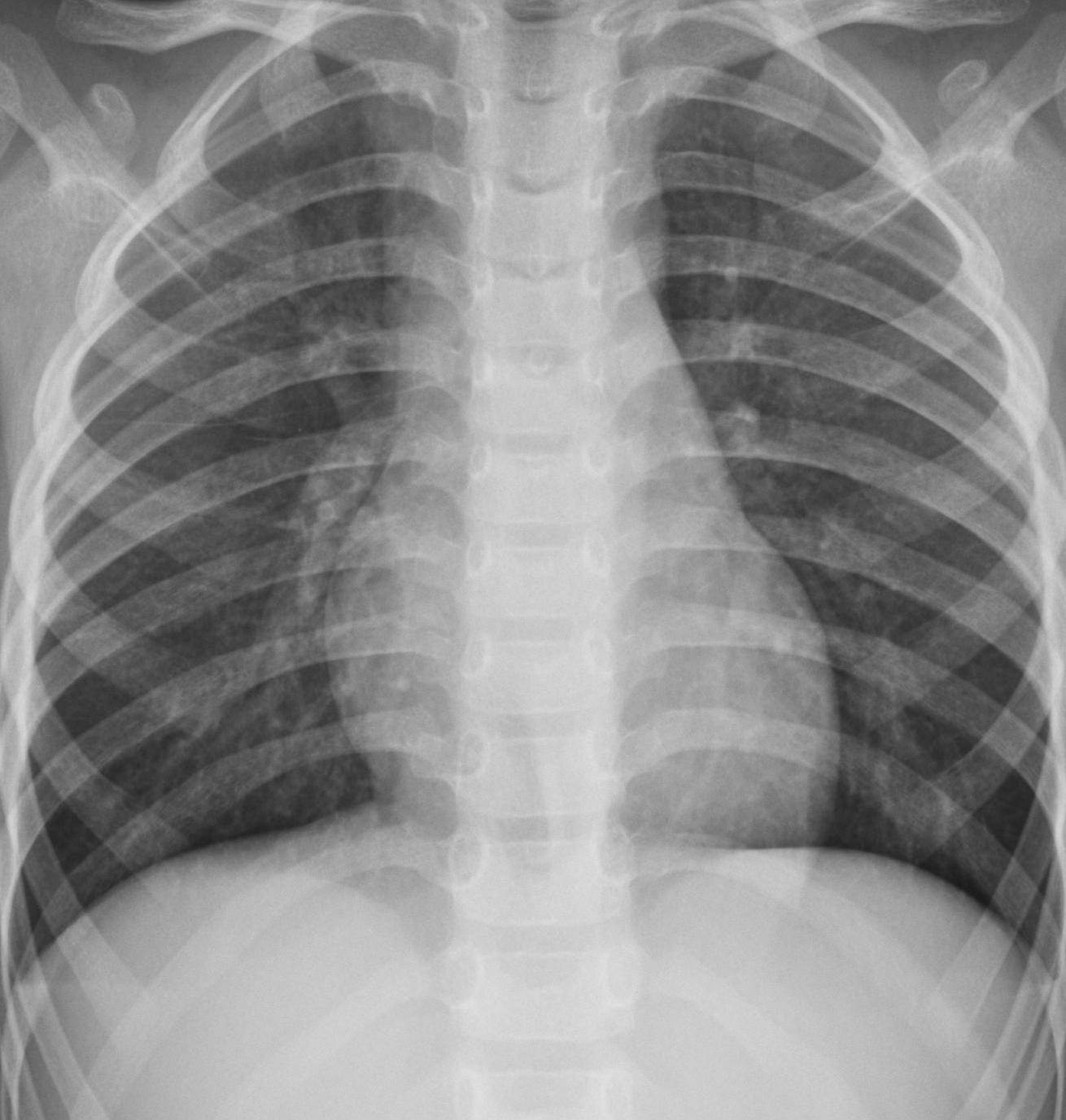}
				\subcaption{}\label{fig:radiografia}
			\end{subfigure}\hfill
\begin{subfigure}[t]{.48\linewidth}
				\includegraphics[width=\linewidth]{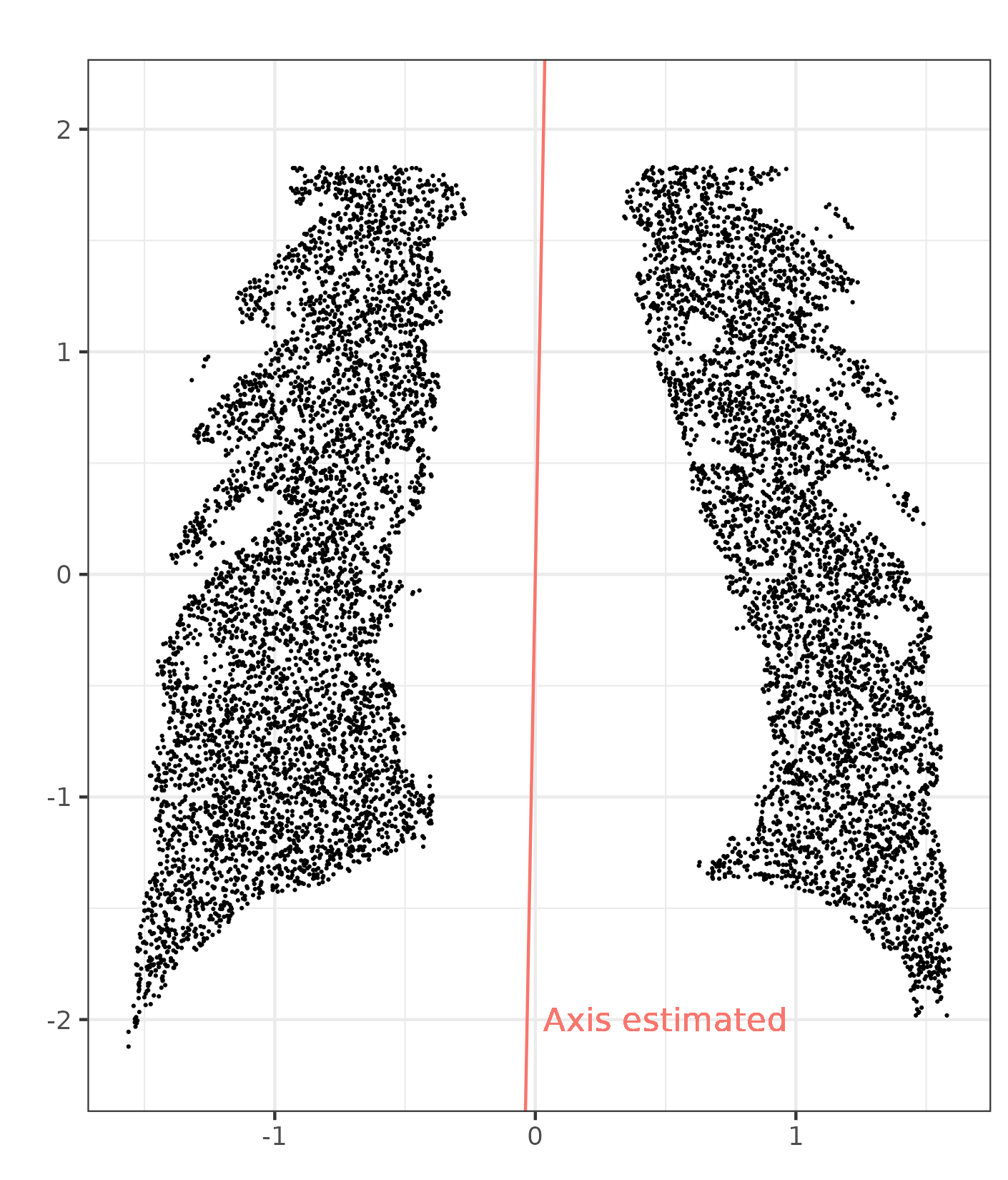}
				\subcaption{}\label{fig:muestra_pulmones}
			\end{subfigure}
			\caption{
            (a) A frontal chest radiograph. The region of interest (the lungs) is highlighted.
            (b) A set of independent random points sampled uniformly from within the lung shape. The red line indicates the axis of symmetry estimated from this point cloud.}

	\end{figure}

    \begin{figure}[!htbp]
        \centering
        \includegraphics[width = \linewidth]{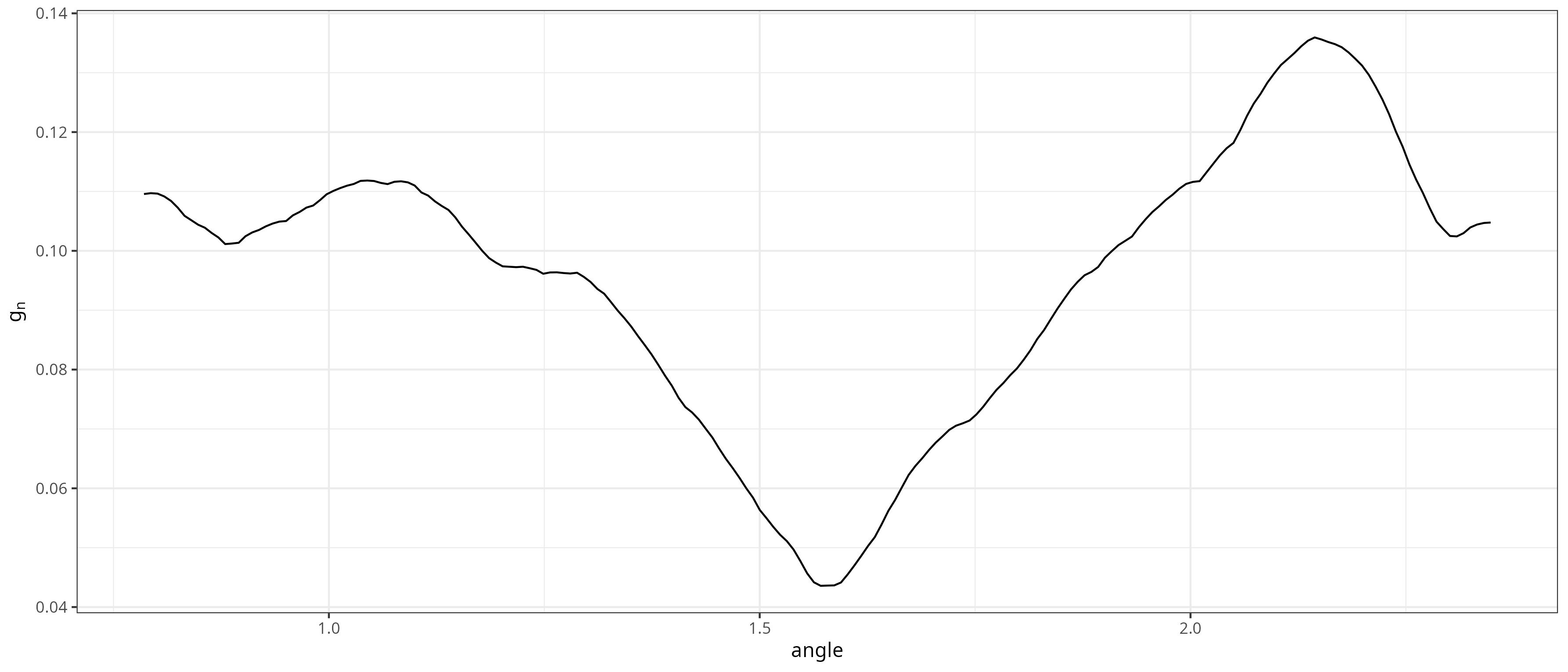}
        \caption{Plot of the estimating function $\widehat g_n$ for the lungs data.}
        \label{fig:gn_lungs}
    \end{figure}

	From the grayscale image in \Cref{fig:radiografia}, we select pixels with normalized intensity below $0.45$ (0 denotes black and 1 white), then sample $n=10000$ pixel locations with added jitter; see \Cref{fig:muestra_pulmones}.

	This yields a sample $\aleph_n=\{X_1,\dots,X_n\}$ from an unknown planar distribution $F_X$. We center it at the sample mean before computing $\widehat g_n$, as in Remark~\ref{rem:center}; the estimated axis passes through that mean. If $X-\mu\eqd R_u(X-\mu)$ for some $u\in\Sd[1]$, where $\mu=\E(X)$, the centered level-set estimator is consistent under the assumptions of Corollary~\ref{cor:consistency-d2} and Remark~\ref{rem:center}. The point cloud is used only to estimate the axis; symmetry features require additional image information.

    \Cref{fig:gn_lungs} shows a well-separated minimum of $\widehat g_n$ near $\pi/2$, in the expected direction of the spinal column. That axis is drawn in red in \Cref{fig:muestra_pulmones}. The sharp minimum suggests a preferred alignment, whereas a flat profile would make the direction less clear. This is a qualitative interpretation: neither the profile nor its minimum is a calibrated measure of uncertainty or pathological asymmetry, which would require separate validation.

\subsection*{Acknowledgments} We thank the referees for their helpful comments and suggestions, which improved both the content and the presentation of the paper. S. Nagy was supported by 
the ERC CZ grant LL2407 of the Ministry of Education, Youth and Sport of the Czech Republic. M. Hernández-Banadik was supported by the Initiation to Research Project of the Sectorial Commission for Scientific Research (project code 22420230100284UD).

	\bibliographystyle{apalike}
	\bibliography{biblio}

\end{document}